\documentclass[a4paper,11pt]{amsart}
\usepackage{amsmath, amsfonts, amssymb, amsthm, amscd}
\usepackage[colorlinks=true, linkcolor=blue, citecolor=blue]{hyperref}
\usepackage{graphicx}
\usepackage{enumerate}
\usepackage{url}
\usepackage{color}
\usepackage[utf8]{inputenc}
\usepackage{tikz}
\usetikzlibrary{arrows}
\usetikzlibrary{patterns} 
\usepackage[a4paper,scale={0.72,0.74},marginratio={1:1},footskip=7mm,headsep=10mm]{geometry}

\newtheorem{theorem}{Theorem}[section]
\newtheorem{lemma}[theorem]{Lemma}
\newtheorem{proposition}[theorem]{Proposition}
\newtheorem{corollary}[theorem]{Corollary}
\newtheorem{remark}[theorem]{Remark}

\newtheorem{conjecture}[theorem]{Conjecture}

\newcommand{\bbE}{{\ensuremath{\mathbb E}} }

\newcommand{\bbN}{{\ensuremath{\mathbb N}} }

\newcommand{\bbP}{{\ensuremath{\mathbb P}} }

\newcommand{\bbR}{{\ensuremath{\mathbb R}} }

\newcommand{\bbZ}{{\ensuremath{\mathbb Z}} }

\newcommand{\cF}{{\ensuremath{\mathcal F}} }

\newcommand{\cN}{{\ensuremath{\mathcal N}} }

\newcommand{\cR}{{\ensuremath{\mathcal R}} }
\newcommand{\cS}{{\ensuremath{\mathcal S}} }

\newcommand{\cZ}{{\ensuremath{\mathcal Z}} }

\newcommand{\gb}{\beta}
\newcommand{\gga}{\gamma}

\newcommand{\gd}{\delta}
\newcommand{\gD}{\Delta}

\newcommand{\gt}{\theta}

\newcommand{\gl}{\lambda}

\newcommand{\gs}{\sigma}

\newcommand{\go}{\omega}
\newcommand{\gO}{\Omega}

\renewcommand{\tilde}{\widetilde}          
\DeclareMathSymbol{\leqslant}{\mathalpha}{AMSa}{"36} 
\DeclareMathSymbol{\geqslant}{\mathalpha}{AMSa}{"3E} 
\DeclareMathSymbol{\eset}{\mathalpha}{AMSb}{"3F}     
\newcommand{\dd}{\text{\rm d}}             
\newcommand{\sumtwo}[2]{\sum_{\substack{#1 \\ #2}}} 

\newcommand{\Z}{\mathbb{Z}}
\newcommand{\N}{\mathbb{N}}

\newcommand{\PEfont}{\mathrm}

\DeclareMathOperator{\cov}{\ensuremath{\PEfont Cov}}

\newcommand\bP{\ensuremath{\mathrm{P}}}
\newcommand\bE{\ensuremath{\mathrm{E}}}

\newcommand\bV{\ensuremath{\mathrm{V}}}

\renewcommand{\epsilon}{\varepsilon}

\newcommand{\ind}{{\sf 1}}

\newcommand{\card}{\mathrm{card}}

\newenvironment{myenumerate}{
\renewcommand{\theenumi}{\arabic{enumi}}
\renewcommand{\labelenumi}{{\rm(\theenumi)}}
\begin{list}{\labelenumi}
{
\setlength{\itemsep}{0.4em}
\setlength{\topsep}{0.5em}
\setlength\leftmargin{2.45em}
\setlength\labelwidth{2.05em}
\setlength{\labelsep}{0.4em}
\usecounter{enumi}
}
}
{\end{list}
}

\renewenvironment{enumerate}{
\begin{myenumerate}}
{\end{myenumerate}}

\newcommand{\beq}{\begin{equation}}
\newcommand{\eeq}{\end{equation}}
\numberwithin{equation}{section}
\newcommand{\ba}{\begin{aligned}}
\newcommand{\ea}{\end{aligned}}

\newcommand{\cst}{\mathrm{(cst.)}}
\newcommand{\sfu}{{\sf u}}
\newcommand{\sfc}{{\sf c}}
\newcommand{\sfD}{{\sf D}}

\newcommand{\geo}{\mathrm{Geo}}

\newcommand{\sfi}{\mathsf{i}}
\newcommand{\sfj}{\mathsf{j}}

\begin{document}
\title[Free energy of the quenched charged polymer]{Variational representation and estimates for the free energy of a quenched charged polymer model}
\author{Julien Poisat}
\begin{abstract} Random walks with a disordered self-interaction potential may be used to model charged polymers. In this paper we consider a one-dimensional and directed version of the charged polymer model that was introduced by Derrida, Griffiths and Higgs. We prove new results for the associated quenched free energy, including a variational formula based on a quenched large deviation principle established by Birkner, Greven and den Hollander. We also take the occasion to (i) provide detailed proofs for state-of-the-art results pointing towards the existence of a freezing transition and (ii) proceed with minor corrections for two results previously obtained by the present author with Caravenna, den Hollander and Pétrélis for the undirected model.
\end{abstract}

\thanks{The author would like to thank an anonymous referee for valuable comments and acknowledges the support of the ANR grant LOCAL (ANR-22-CE40-0012).}
 
\date{\today}

\maketitle

\tableofcontents

\section{Introduction}
This paper deals with a particular class of random polymer models called \emph{charged polymers}. Polymers are large molecules, either natural or artificial, which result from the bonding of many constituent units (atoms or group of atoms) called \emph{monomers}. We shall hereafter restrict to flexible \emph{chain-like} structures. The spatial configuration (a.k.a.\ \emph{conformation}) of such macromolecules may drastically change upon a small change in parameter like the temperature (e.g.\ the denaturation of DNA or the collapse transition of a polymer in a poor solvent). \emph{Random polymer models} usually refer to a class of \emph{discrete probabilistic models} in which (i) the possible conformations are given by  a subset of \emph{finite lattice paths} and (ii) the \emph{polymer measure} (for a given number of monomers) is a reference measure on such paths (e.g.\ the simple random walk measure) weighted by a Gibbs factor (the exponential of minus the energy or Hamiltonian function), following the formalism of equilibrium statistical mechanics. We refer to~\cite{Comets-book,dH09,Gia07,Gia11} for a review of some of these models.
\par The topic of polymers is closely related to the phenomenon of \emph{anomalous diffusion}. Unweighted walks, also referred to as \emph{ideal} chains, lead to the standard \emph{diffusive} scaling but are not realistic in the polymer context. There are different ways of taking into account interactions and excluded volume effects, reflecting the fact that distinct monomers cannot occupy the same position and may interact with themselves or their medium, see e.g.~\cite{dH09}. While standard excluded volume effects, like self-intersection penalization (see~\cite{MadrasSladeBook13}), typically lead to \emph{super-diffusive} (or extended) behaviour (at least in low dimensions), other models can also exhibit \emph{sub-diffusive} (or collapsed) behaviour. This is the case for \emph{charged polymers}, which is the focus of the present paper.
\par In a \emph{charged polymer} (or polyelectrolyte), each monomer carries an electric charge that interacts with the charges on the other monomers: monomers with alike charges repel each other while monomers with opposite charges attract each other. This leads to complex interactions, especially when the charge sequence (as it is read along the polymer chain) is itself \emph{disordered}. Following~\cite{KantorKardar1991}, the (free) polymer chain is modeled by the simple random walk on $\bbZ^d$, the charges are independent and identically distributed (i.i.d.) real-valued random variables, and interactions happen between pairs of monomers at \emph{self-intersections} only (two-body short-range interactions). For the \emph{annealed} (i.e. averaged out charges) version of the model, a phase transition between a \emph{collapsed} phase (attractive interactions prevail) and an \emph{extended} phase (repulsive interactions prevail) in the inverse temperature vs.\ charge bias phase diagram has been established under exponential moment conditions on the charge distribution, see \cite{BdHP2018,CdHPP16}. The \emph{quenched} (frozen charges) version however remains largely open, apart from partial results showing extended behavior of the chain under a large charge bias condition, see \cite{CdHPP16}.
\par In the present paper we consider the one-dimensional \emph{directed}\footnote{The word \emph{directed} can be misleading. In the context of polymer models, it usually refers to the use of a $(1+d)$-dimensional random walk, that is a walk which moves deterministically in the first dimension and according to simple random walk in the remaining $d$ dimensions. It has a different meaning in this paper.} version of this model, which was originally introduced by Derrida, Griffiths and Higgs, see~\cite{DeGrHi92,DeHi94}. In that model, the random walk may only move to the right or stay put (instead of moving to the left) with equal probability, at each unit step. The aforementioned authors argued in favor of a \emph{freezing transition} between a collapsed and an extended phase, with some results being supported by heuristics and other results that can be made rigorous. The main purpose of this paper is to review previous results and provide new estimates for the quenched free energy associated to this model. In particular, we show that the latter fits into a larger class of one-dimensional disordered models (including disordered pinning and copolymer models) to which the Birkner-Greven-den Hollander \emph{quenched} (or conditional) \emph{large deviation principle} (see~\cite{B08,BGDH10,BGDH10-add}) may be applied to obtain a variational formula for the free energy. This is the content of Section~\ref{sec:DGH-model}. We also take the occasion to correct two propositions from~\cite[Appendix D]{CdHPP16}  for the \emph{undirected} model, see Appendix~\ref{sec:corrections}.\\

{\bf Notation.} We denote by $\bbN_0=\{0,1,2,\ldots\}$ and $\bbN=\{1,2,\ldots\}$ the sets of non-negative and positive integers, respectively. We write $a\wedge b = \min(a,b)$ and $a\vee b = \max(a,b)$.\\

{\bf Outline.} We recall the definition of the model in Section~\ref{sec:DGH-model} and review previous results as well as conjectures made in~\cite{DeGrHi92} in Section~\ref{sec:predictions-and-sota}. Our results are contained in Sections~\ref{sec:var-rep-FE} to~\ref{sec:low-temp-FE}. In Section~\ref{sec:var-rep-FE}, we first show that the (infinite volume) quenched free energy exists and then provide a variational formula for it using the so-called Large Deviation Principle (LDP) \emph{for words drawn in a letter sequence}, see~\cite{B08,BGDH10,BGDH10-add}. Finally, we obtain some bounds on the free energy at high and low temperatures, in Sections~\ref{sec:high-temp-FE} and~\ref{sec:low-temp-FE} respectively.\\

\par Before we continue, let us distinguish the old material from novel work in the present paper. Proposition~\ref{pr:DGH}, Proposition~\ref{pr:DGH2} and Lemma~\ref{lem:DGH} are taken from~\cite{DeGrHi92}, but we take the occasion to write full proofs, obtaining thereby a (marginally) improved lower bound in Proposition~\ref{pr:DGH2}. The idea behind Proposition~\ref{pr:low-temp-two-block-dist} goes back to~\cite{DeGrHi92}, but we obtain a more precise statement by writing a detailed proof. We have not seen the results in Theorem~\ref{thm:que-fe} and Proposition~\ref{pr:prop-fe} displayed elsewhere, although they follow from standard arguments. Also, the argument behind the lower bound in Item~\ref{item:prop-fe2} of Proposition~\ref{pr:prop-fe} appeared in~\cite{DeGrHi92}. The other results in this section are new, to the best of the author's knowledge.\\

\section{The Derrida-Griffiths-Higgs directed model}
\label{sec:DGH-model}
%
%
%
\par We introduce two (slightly different but equivalent) formulations for the model. The formulation in Section~\ref{sec:origin-conv} corresponds to that in~\cite{DeGrHi92} whereas that in Section~\ref{sec:alt-conv} looks closer to the formulation used in~\cite{BdHP2018,CdHPP16}. Finally, we observe in Section~\ref{sec:renewal_times} that the model fits into a class of (disordered) statistical-mechanics systems built on renewal sequences.
\subsection{Original formulation}
\label{sec:origin-conv}
We assume that the binary charges $\go = (\go_i)_{i\ge 1}$ are i.i.d.\ and follow the unbiased charge distribution $\bbP(\go_1=1) = \bbP(\go_1=-1) = 1/2$, unless stated otherwise. When summing consecutive charges we shall adopt the notation\footnote{We sometimes write $\gO(a,b]$ instead of $\gO_{(a,b]}$ for questions of readability.}
\beq
\label{eq:def-gO-a-b}
\gO_{(a,b]} = \sum_{a<i\le b} \go_i, \qquad 0\le a <b,
\eeq
and for conciseness we write $\gO_n$ instead of $\gO_{(0,n]}$. The set of allowed paths (i.e.\ the polymer configurations) is given by
\beq
\cS_n = \{(s_i)_{i=0}^n\colon s_0=0,\ s_1=1,\ s_{i+1}-s_{i} \in \{0,1\},\ 1\le i< n\}, \qquad n\ge 2,
\eeq
with $\cS_0 := \{0\}$ and $\cS_1 := \{(0,1)\}$.
The fact that the first step is fixed to one in the above definition is only a matter of convention. In the polymer interpretation, the $i$-th monomer in the chain has position $s_i$ and bears a charge $\go_i$. The Hamiltonian\footnote{Limit theorems and strong approximations for the charged polymer Hamiltonian were derived in~\cite{GPS16,HuK10}. The focus in this paper is different.} is then defined as
\beq
\label{eq:ham-DGH}
H_n^\go(s) = \sum_{1\le i< j \le n} \go_i \go_j \ind_{\{s_i=s_j\}}, \qquad s\in \cS_n,
\eeq
so that the quenched partition function writes:
\beq
\label{eq:pf-DGH}
Z_{n}^{\gb,\go} = \sum_{s\in \cS_n} \exp(-\gb H_n^\go(s)),
\eeq
where $\gb\ge 0$ is the \emph{inverse temperature}. For $n\in\{0,1\}$ the sum in~\eqref{eq:ham-DGH} is empty, hence $H_0^\go=H_1^\go=0$ and $Z_{0}^{\gb,\go}=Z_{1}^{\gb,\go}=1$. The (quenched) charged polymer of size $n\ge 1$ is a random path $(S_i)_{i=0}^n$ distributed as $\bP_{n}^{\gb,\go}$, the quenched polymer measure, that is defined by
\beq
\bP_{n}^{\gb,\go}((S_i)_{i=0}^n = s) = \frac{\exp(-\gb H_n^\go(s))}{Z_{n}^{\gb,\go}}, \qquad s\in \cS_n,
\eeq
and we denote by $\bE_n^{\gb,\go}$ (resp.\ $\bV_n^{\gb,\go}$) the corresponding expectation (resp.\ variance). For \emph{partial} partition functions, we write 
\beq
\label{eq:def-partial-part-fct}
Z_{n}^{\gb,\go}(A) = \sum_{s\in A} \exp(-\gb H_n^\go(s)), \qquad A\subseteq \cS_n.
\eeq 
In the sequel we shall also use the notation
\beq
\label{eq:def-shifted-part-fct}
Z_{(a,b]}^{\gb,\go} = Z_{b-a}^{\gb,\theta^a\go}, \qquad 0\le a \le b,
\eeq
where $\theta$ is the shift operator, i.e.\ $(\theta\go)_i = \go_{i+1}$. If $a=b$ then $(a,b]$ is the empty set and the partition function in~\eqref{eq:def-shifted-part-fct} equals $Z_0^{\gb,\theta^a\go}=1$.
\subsection{Alternative formulation}
\label{sec:alt-conv}
Let $(S_i)_{i\ge 0}$ be a directed random walk on the set of natural integers such that $S_0 = 0$, $S_1=1$ and $(\gD S_i)_{i\ge 1} := (S_{i+1}- S_i)_{i\ge 1}$ is a sequence of i.i.d.\ random variables uniformly distributed on $\{0,1\}$ (corresponding to \emph{folded} or \emph{stretched} monomers). Define the quenched partition function as
\beq
\label{eq:part-fct-conv2}
\bar Z_{n}^{\gb,\go} = \bE \Big[\exp\Big(-\gb \sum_{1\le i, j\le n} \go_i \go_j \ind_{\{S_i=S_j\}}\Big)\Big], \qquad n\ge 1,
\eeq
where $\bE$ is the expectation with respect to the law $\bP$ of the random walk $S$. It will be convenient in the sequel to set the convention $\bar Z_{0}^{\gb,\go}=2$. For \emph{partial} partition functions, we write for any event $A$ (that is measurable w.r.t.\ the random walk):
\beq
\bar Z_{n}^{\gb,\go}(A) = \bE \Big[\exp\Big(-\gb \sum_{1\le i, j\le n} \go_i \go_j \ind_{\{S_i=S_j\}}\Big) \ind_A\Big], \qquad n\ge 1.
\eeq
It was already noted (see e.g.~\cite{DeGrHi92}) that the partition function in~\eqref{eq:part-fct-conv2} can be rewritten
\beq
\label{eq:bar-pf}
\bar Z_{n}^{\gb,\go} = \bE \Big[\exp\Big(-\gb \sum_{x\ge 1} \gO_n(x)^2\Big)\Big], \qquad n\in\bbN,
\eeq
where
\beq
\label{eq:def-Omega-n-x}
\gO_n(x) = \sum_{1\le i \le n} \go_i \ind_{\{S_i = x\}}
\eeq
is the cumulated charge at site $x\in \bbN$. The two formulations are actually equivalent. Indeed, the equality (recall the definition of the Hamiltonian in~\eqref{eq:ham-DGH})
\beq
\sum_{1\le i, j\le n} \go_i \go_j \ind_{\{S_i=S_j\}} = n + 2 H_n^\go(S)
\eeq
(that holds for \emph{binary} charges) implies that for every $n\in \bbN_0$,
\beq
\bar Z_{n}^{\gb,\go} = 2^{1-n}e^{-\gb n}Z_{n}^{2\gb,\go}.
\eeq
Consequently, the polymer measures coincide up to a different scaling of the inverse temperature:
\beq
\label{eq:diff-scaling-polymer-meas}
\bar \bP_{n}^{\gb,\go} := \frac{1}{\bar Z_{n}^{\gb,\go}}\bE \Big[\exp\Big(-\gb \sum_{x\ge 1} \gO_n(x)^2\Big)\ind_{\{\cdot\}}\Big] = \bP_{n}^{2\gb,\go},
\eeq
and the corresponding versions of the quenched free energy are easily connected one to another:
\beq
\label{eq:equ-fe}
\bar F_{\rm que}(\gb) = -\gb - \log 2 + F_{\rm que}(2\gb),
\eeq
where
\beq
\label{eq:def-inf-vol-fe}
F_{\rm que}(\gb) := \lim_{n\to \infty} (1/n) \log Z_n^{\gb, \go}
\quad \text{and} \quad
\bar F_{\rm que}(\gb) := \lim_{n\to \infty} (1/n) \log \bar Z_n^{\gb, \go},
\eeq
provided that these limits exist. We will come back to the issue of existence later in the paper and prove that such limits actually exist in the $\bbP$-a.s.\  and in the $L^1(\bbP)$ sense, see Theorem~\ref{thm:que-fe} below. 

\subsection{Renewal times}
\label{sec:renewal_times}
We set $\tau_0 := 0$ and $\tau_i = \inf\{k>\tau_{i-1}\colon \gD S_k = 1\}$ for every $i\ge 1$, or in other words, $\tau = \{\tau_i\}_{i\ge 0} = \{k\ge 0\colon \gD S_k = 1\}$. The initial point $\tau_0 = 0$ reflects the fact that $S_0=0$ and $S_1=1$. Note that $\tau$ is a renewal process with $\geo(1/2)$ inter-arrival distribution on $\bbN$ and that
\beq
S_k = x \quad \text{if and only if} \quad \tau_{x-1} < k \le \tau_x, \qquad k\ge 1,\ x\ge 1.
\eeq
Then, the (quenched) partition function can be expressed in terms of this renewal process directly, instead of the random walk:
\beq
\label{eq:bar-pf2}
\bar Z_{n}^{\gb,\go} = \bE \Big[\exp\Big(-\gb \sum_{i\ge 1} \gO(\tau_{i-1}\wedge n, \tau_i \wedge n]^2\Big)\Big],
\eeq
using the notation in~\eqref{eq:def-gO-a-b} and the convention that $\gO(\emptyset) = 0$ when $\tau_{i-1}\wedge n = \tau_i \wedge n\, (= n)$.
In this way we see that this directed model fits into a broader class of (disordered) statistical mechanics model that are built on renewal sequences and that have witnessed remarkable progress quite recently, such as the random pinning model and the copolymer model, see e.g.~\cite{dH09,Gia07,Gia11} and references therein. We shall come back to this observation in Section~\ref{sec:var-rep-FE}.\\

\section{Predictions and state of the art}
\label{sec:predictions-and-sota}
In this section we recall a series of predictions and observations originally made in~\cite{DeGrHi92}.

\subsection{Freezing transition}
Let us define
\beq
\label{eq:def-p-i-n}
p_{i,n}^{\gb,\go} := \bE_n^{\gb,\go}(\gD S_i), \qquad \gD S_i = S_{i+1}- S_i\in\{0,1\},\quad  1\le i<n,
\eeq
as well as the quenched \emph{empirical cumulative distribution function}
\beq
\cF_n^{\gb,\go}(p) = \frac 1{n-1} \sum_{1\le i< n} \ind\{p_{i,n}^{\gb,\go} \le p\}, \qquad p\in [0,1], \quad n\ge 2.
\eeq
Derrida, Griffiths and Higgs (DGH) noticed, based on numerical simulations, that {\it ``many $p_i$'s have a rapid and nonmonotonic temperature variation, not unlike the ``chaotic" behaviour of the local magnetization $m_i$ in a spin glass"}, see \cite{DeGrHi92}, .
More precisely, they formulated the following:
\begin{conjecture}
\label{conj:DGH}
(Self-averaging and ``weak freezing transition'')
\begin{enumerate}
\item\label{item:DGH1}  The specific heat (i.e.\ the derivative of the finite volume free energy\footnote{The (quenched) finite volume free energy, for a system of size $n\ge 1$, is the quantity $(1/n) \log Z_n^{\gb, \go}$ appearing in \eqref{eq:def-inf-vol-fe}}) converges $\bbP$-almost-surely, as $n\to\infty$, to a non-random limit. 
\item\label{item:DGH2} The empirical cumulative distribution function $\cF_n^{\gb,\go}(p)$ converges $\bbP$-almost-surely, as $n\to\infty$, to a non-random limit $\cF_{\gb}(p)$.
\item\label{item:DGH3} There exists a critical inverse temperature $\gb_c\in(0,\infty)$ such that:
\begin{itemize}
\item For every $\gb< \gb_c$, there exists $p_{\rm min}(\gb)\in(0,1)$ such that $\cF_\gb(p) =0$ for every $p<p_{\rm min}(\gb)$ (high-temperature regime).
\item For every $\gb>\gb_c$, there exists an exponent $\gga(\gb)>0$ such that $\cF_\gb(p) \sim \cst\, p ^{\gga(\gb)}$ as $p\to 0$ (low-temperature regime).
\end{itemize}
\end{enumerate}
\end{conjecture}
\begin{remark}
The authors in~\cite[Eq.\ (14)]{DeGrHi92} provide heuristic arguments showing that the exponent $\gamma(\gb)$ appearing in Item~\ref{item:DGH3} of Conjecture~\ref{conj:DGH} continuously varies in $\gb$ and solves
\beq
\sum_{n\ge 1} 2^{n/\gamma} e^{-\frac12{\gb n^2}} = 1.
\eeq
\end{remark}
In~\cite{DeGrHi92}, the authors established:
\begin{itemize}
\item the presence of a \emph{high-temperature regime} for $\gb < \log(\frac{1+\sqrt{5}}{2})$ and general charge sequences, see Section~\ref{sec:high-temp-reg};
\item the presence of a \emph{low-temperature regime} for a particular di-block $n$-dependent charge sequence, namely (for a system of size $2n$)
\beq
\label{eq:two-block-charges}
\go_i^{(2n)} =
\begin{cases}
1 & \text{if} \quad 1\le i < n,\\
-1 & \text{if} \quad n\le i <2n,
\end{cases}
\eeq
see Section~\ref{sec:low-temp-reg}.
\end{itemize}
The remaining points in Conjecture~\ref{conj:DGH}, including the presence of a low-temperature phase for the binary (or other) charge distribution, are still open, to the best of our knowledge. Let us stress that it is not clear to us whether the freezing transition could be accessed through the analytical properties of the free energy function, as it usually happens for phase transitions in statistical mechanics (which correspond to the lack of analyticity in the free energy).\\

\par In the rest of the section, we provide full proofs for some of the results stated in~\cite{DeGrHi92}.
\subsection{Recursion relation}
Recall the definition of $\tau$ in Section~\ref{sec:renewal_times}. By decomposing a random walk path according to $\max\{0\le k< n \colon k\in \tau\}$, that is the last renewal point to be found (strictly) before monomer $n\ge 1$, we readily get the following recursive relation:
\beq
Z_{n}^{\gb,\go} = \sum_{0\le k < n} Z_{k}^{\gb,\go} \exp(\tfrac{\gb}{2}[n-k - \gO_{(k,n]}^2]),
\eeq
or equivalently,
\beq
\bar Z_{n}^{\gb,\go} = \sum_{0\le k < n} \bar Z_{k}^{\gb,\go} 2^{k-n}\exp(-\gb \gO_{(k,n]}^2).
\eeq
As noticed in~\cite{DeGrHi92} this relation allows for fast numerical simulation.
\subsection{The high temperature regime}
\label{sec:high-temp-reg}
Recall the definition of $p_{i,n}^{\gb,\go}$ in \eqref{eq:def-p-i-n}.
\begin{proposition}[See \cite{DeGrHi92}]
\label{pr:DGH}
For every $1\le i< n$ and every charge sequence $\go$, we have
\beq
\label{eq:DGH}
p_{i,n}^{\gb,\go} - \frac12 \sim \gb  \sumtwo{0<u\le i}{i<v\le n}\go_u \go_v 2^{u-v}, \qquad \gb\to 0.
\eeq
\end{proposition}
\begin{proof}[Proof of Proposition~\ref{pr:DGH}]
Let $1\le i < n$ and consider the ratio:
\beq
\frac{\bP_{n}^{\gb,\go}(\gD S_i = 1)}{\bP_{n}^{\gb,\go}(\gD S_i = 0)}=
\frac{Z_{n}^{\gb,\go}(\gD S_i = 1)}{Z_{n}^{\gb,\go}(\gD S_i = 0)}.
\eeq
With a slight abuse of notation, we have applied the notation in~\eqref{eq:def-partial-part-fct} for partial partition functions to events $A\in \gs(S_i, 0\le i<n)$.
Let us first deal with the numerator. We first note that:
\beq
\label{eq=DGH-spin1a}
Z_{n}^{\gb,\go}(\gD S_i = 1) = Z_{i}^{\gb,\go}Z_{(i,n]}^{\gb,\go}.
\eeq
By decomposing a path according to $k:= \max\{[0,i)\cap \tau\}$ and $\ell:=\min\{(i,n)\cap \tau\}$ (by convention we set $\ell := n$ if the latter set is empty), we obtain
\beq
\label{eq=DGH-spin1}
Z_{n}^{\gb,\go}(\gD S_i = 1) = \sumtwo{0\le k< i}{i<\ell\le n}
Z_{k}^{\gb,\go}
e^{\frac{\gb}{2}(\ell-k-\gO_{(k,i]}^2 - \gO_{(i,\ell]}^2)}Z_{(\ell,n]}^{\gb,\go},
\eeq
with the convention that the rightmost partition function in the line above is to be understood as equal to one if $\ell = n$. As for the denominator, a similar decomposition gives:
\beq
\label{eq=DGH-spin0}
Z_{n}^{\gb,\go}(\gD S_i = 0) = \sumtwo{0\le k< i}{i<\ell\le n}
Z_{k}^{\gb,\go}
e^{\frac{\gb}{2}(\ell-k-\gO_{(k,\ell]}^2)}Z_{(\ell,n]}^{\gb,\go}.
\eeq
Comparing both expressions, we observe that
\beq
\gO_{(k,\ell]}^2 - \gO_{(k,i]}^2 - \gO_{(i,\ell]}^2 = 2 \sumtwo{k< u \le i}{i<v\le\ell} \go_u \go_v.
\eeq
Performing a small-$\beta$ expansion, and noting that $Z_{n}^{\gb,\go}$ converges to $2^{(n-1)\vee 0}$ as $\gb\to 0$, we obtain (the case $k=0$ or $\ell=n$ needs care): 
\beq
p_{i,n}^{\gb,\go} - \frac12 \sim \frac \gb 2 \sumtwo{0\le k< i}{i<\ell\le n}
 2^{k\vee 1 - \ell \wedge (n-1)} \sumtwo{k< u \le i}{i<v\le\ell} \go_u \go_v, \qquad \gb\to 0.
\eeq
We conclude by noting that
\beq
\sumtwo{0\le k< i}{i<\ell\le n} 2^{k\vee 1 - \ell \wedge (n-1)} \sumtwo{k< u \le i}{i<v\le\ell} \go_u \go_v = \sumtwo{0<u\le i}{i<v\le n}\go_u \go_v 2^{1+u-v}.
\eeq
\end{proof}
Let us make a few comments on Proposition~\ref{pr:DGH}. First, we observe that the series in the r.h.s.\ of \eqref{eq:DGH} converges, under mild assumptions on the charge sequence (say boundedness), as $n\to\infty$:
\beq
\lim_{n\to \infty} \sumtwo{0<u\le i}{i<v\le n}\go_u \go_v 2^{u-v} = \sum_{0<u\le i< v}\go_u \go_v 2^{u-v}.
\eeq
This shows that far-away charges have a decreasing influence on the state of a given monomer, at high temperatures. However, it is not clear whether the limits $\gb\to 0$ and $n\to\infty$ may be interchanged. Using the same proof as above but replacing the small-$\gb$ expansion by the inequality $e^x\ge 1+x$, one can interchange the limits and obtain:
\beq
\limsup_{\gb\to 0} \frac{1}{\gb}\Big(p_{i,n}^{\gb,\go} - \frac12\Big) \le   \sumtwo{0<u\le i}{i<v\le n}\go_u \go_v 2^{u-v},
\eeq
uniformly in $1\le i < n$. This is however not the desired direction, in view of Item~\ref{item:DGH3} in Conjecture~\ref{conj:DGH}. Fortunately, DGH noticed the following:
\begin{proposition}[See~\cite{DeGrHi92}]
\label{pr:DGH2}
For every $1\le i< n$ and every charge sequence $\go$,
\beq
p_{i,n}^{\gb,\go} \ge \left(1+ \frac{e^{\gb/2}+e^{\gb}[1+e^{-\gb}]^{-1}}{1-e^\gb[1+e^{-\gb}]^{-1}}\right)^{-1},
\eeq
which is positive for every $\gb < \log(\frac{1+\sqrt{5}}{2})$.
\end{proposition}
The proof relies on the following lemma:
\begin{lemma}[See~\cite{DeGrHi92}]
\label{lem:DGH}
 For every $n\ge 1$, $Z_{n}^{\gb,\go} \ge Z_{n-1}^{\gb,\go}$ and for every $n\ge 2$, $Z_{n}^{\gb,\go}\ge (1+e^{-\gb})Z_{n-2}^{\gb,\go}$.
\end{lemma}
Let us stress that Lemma~\ref{lem:DGH} holds for {\it any} charge sequence $\go$.
\begin{proof}[Proof of Lemma~\ref{lem:DGH}]
Recall the expression of the partition function in~\eqref{eq:pf-DGH}. The proof follows by restricting $\cS_n$ to $\{s\in \cS_n \colon s_n - s_{n-1}= 1\}$ for the first inequality, and $\{s\in \cS_n \colon s_{n-1} - s_{n-2} = 1\}$
for the second one.
\end{proof}
\begin{proof}[Proof of Proposition~\ref{pr:DGH2}]
Starting from~\eqref{eq=DGH-spin0} and ignoring the term $\gO_{(k,\ell]}^2$ in the exponential, we get:
\beq
\label{eq:DGH2-aux1}
Z_{n}^{\gb,\go}(\gD S_i = 0) \le \sum_{0\le k< i}
Z_{k}^{\gb,\go} e^{\frac{\gb}{2}(i-k)}
\sum_{i<\ell\le n}
e^{\frac{\gb}{2}(\ell-i)}Z_{(\ell,n]}^{\gb,\go}.
\eeq
Using Lemma~\ref{lem:DGH} repeatedly, we have for every $0\le u \le (n-1)/2$,
\beq
\max(Z_{n-2u}^{\gb,\go},Z_{n-2u-1}^{\gb,\go}) \le (1+e^{-\gb})^{-u} Z_{n}^{\gb,\go}.
\eeq
Plugging this estimate into the sum over $k$ and splitting it between even and odd values of $i-k$ (say $i-k = 2u+1$ for some $u\ge0$ or $i-k = 2u$ for some $u\ge1$) we obtain:
\beq
\label{eq:DGH2-aux2}
\ba
\sum_{0\le k< i}
Z_{k}^{\gb,\go} e^{\frac{\gb}{2}(i-k)} &\le Z_i^{\gb,\go} \Big[\sum_{u\ge 0} (1+e^{-\gb})^{-u}e^{(2u+1)\gb/2} + \sum_{u\ge 1} (1+e^{-\gb})^{-u}e^{(2u)\gb/2}  \Big]\\
&= Z_i^{\gb,\go} \Big[\frac{e^{\gb/2}+e^{\gb}[1+e^{-\gb}]^{-1}}{1-e^\gb[1+e^{-\gb}]^{-1}} \Big],
\ea
\eeq
as soon as $\gb < \log(\frac{1+\sqrt{5}}{2})$, so that the sum over $u$ is finite.
Similarly,
\beq
\label{eq:DGH2-aux3}
\sum_{i<\ell\le n}
e^{\frac{\gb}{2}(\ell-i)}Z_{(\ell,n]}^{\gb,\go} \le Z_{(i,n]}^{\gb,\go} \Big[\frac{e^{\gb/2}+e^{\gb}[1+e^{-\gb}]^{-1}}{1-e^\gb[1+e^{-\gb}]^{-1}}\Big].
\eeq
Combining~\eqref{eq=DGH-spin1a}, \eqref{eq:DGH2-aux1}, \eqref{eq:DGH2-aux2} and \eqref{eq:DGH2-aux3} leads to
\beq
Z_{n}^{\gb,\go}(\gD S_i = 0) \le \Big[\frac{e^{\gb/2}+e^{\gb}[1+e^{-\gb}]^{-1}}{1-e^\gb[1+e^{-\gb}]^{-1}} \Big]^2 \times Z_{n}^{\gb,\go}(\gD S_i = 1),
\eeq
which completes the proof.
\end{proof}
\begin{remark}
Although we follow the same line of proof, the lower bound in Proposition~\ref{pr:DGH2} is (slightly) better than the one stated in~\cite{DeGrHi92}. However, none of these bounds is close to optimal, as the limit $\gb\to 0$ shows.
\end{remark}
\subsection{The low temperature regime}
\label{sec:low-temp-reg}
Although the existence of a low-temperature regime remained at a heuristic level for the random binary charges, \cite{DeGrHi92} shortly argued that for the special charge distribution defined in~\eqref{eq:two-block-charges} (a.k.a.\ di-block polymer), the chain collapses at inverse temperatures $\gb>2\gb_0$, where
\beq
\label{eq:def-beta0}
\gb_0 := \inf\Big\{\gb >0\colon U(\gb):=\sum_{t\in \bbN} e^{-\gb t^2} < 1\Big\},
\eeq
meaning that ``{\it almost all of the monomers are on one site, as the cancellation of positive and negative charges minimizes the energy}'', see~\cite{DeGrHi92}. In this section, we provide the necessary details of this argument so as to obtain a more quantitative statement. To this purpose, let us define
\beq
\ba
\sfi(n) &= \max([0,n)\cap \tau)\\
\sfj(n) &= \min([n,2n]\cap \tau) - n,
\ea
\eeq
with the convention $\sfj(n)= n$ if $[n,2n]\cap \tau=\emptyset$.
\begin{proposition}
\label{pr:low-temp-two-block-dist}
Recall the definitions of $\gb_0$ and $U(\gb)$ in~\eqref{eq:def-beta0}.
Assume that the charge sequence is as in~\eqref{eq:two-block-charges} and that $\gb>\gb_0$. Then, there exists $C(\gb)\in(0,1)$ such that for every $0\le i <n$ and $0\le j \le n$,
\beq
\bar \bP_{2n}^{\gb,\go}(\sfi(n)=i) \le \frac{C(\gb)^i}{1-U(\gb)},\qquad 
\bar \bP_{2n}^{\gb,\go}(\sfj(n)=j) \le \frac{C(\gb)^{n-j}}{1-U(\gb)}
\eeq
and for every $M\ge 1$,
\beq
\bar \bP_{2n}^{\gb,\go}(|\sfi(n)+\sfj(n)-n|\ge M) \le \frac{e^{-\gb M^2}}{[1-U(\gb)]^2}.
\eeq
\end{proposition}
\begin{corollary}
Under the same assumptions as Proposition~\ref{pr:low-temp-two-block-dist}, we have
\beq
p_{i,2n}^{2\gb,\go} \le \frac{C(\gb)^{i \wedge (2n-i)}}{[1-U(\gb)][1-C(\gb)]}, \qquad 1\le i\le 2n.
\eeq
\end{corollary}
By~\eqref{eq:diff-scaling-polymer-meas}, the threshold value $\gb_0$ in Proposition~\ref{pr:low-temp-two-block-dist} indeed corresponds to $2\gb_0$ in~\cite{DeGrHi92}.
We will see during the proof that the rate of decay $C(\gb)$ can be expressed in terms of the free energy of the \emph{weakly self-avoiding walk} (defined slightly below).
Proposition~\ref{pr:low-temp-two-block-dist} shows the presence of a low-temperature regime (for this particular charge sequence) where (i) the polymer is very close to its \emph{fully collapsed} state (corresponding to $\sfi(n) = 0$ and $\sfj(n)=n$) and (ii) the cumulated charge in the folded piece of polymer around monomer $n$ (equal to $\sfi(n)+\sfj(n)-n$) has a Gaussian tail. Before starting the proof, let us introduce some additional notation. We write
\beq
\label{eq:wsaw}
\bar Z_{n}^{\gb,+} := \bE \Big[\exp\Big(-\gb \sum_{1\le i, j\le n} \ind_{\{S_i=S_j\}}\Big)\Big], \qquad n\ge 1,
\eeq
that is the partition function displayed in~\eqref{eq:part-fct-conv2} when all the charges equal one (a.k.a.\ the partition function of the \emph{weakly self-avoiding walk}). It can be rewritten as
\beq
\label{eq:wsaw2}
\bar Z_{n}^{\gb,+} = \bE \Big[\exp\Big(-\gb \sum_{x\ge 1} \ell_n(x)^2\Big)\Big], \qquad \text{with}\quad
\ell_n(x):= \{1\le k \le n\colon S_k =x\},
\eeq
with the convention $\bar Z_{0}^{\gb,+}=2$, in order to be consistent with Section~\ref{sec:alt-conv}. We shall use the following lemma, the proof of which is deferred at the end of this section. Recall the definition of $\gb_0$ in~\eqref{eq:def-beta0}.
\begin{lemma}
\label{lem:wsaw}
The following limit exists
\beq
- \bar F_+(\gb) := \lim_{n\to\infty} \frac1n \log \bar Z_{n}^{\gb,+} = \sup_{n\to \infty} \frac1n \log(\tfrac{1}{2}\bar Z_{n}^{\gb,+}).
\eeq
Moreover, $\bar F_+(\gb) > \log 2$ if and only if $\gb > \gb_0$.
\end{lemma}
\begin{proof}[Proof of Proposition~\ref{pr:low-temp-two-block-dist}]
We first observe that
for every $0\le i < n$ and $0\le j\le n$,
\beq
\bar Z_{2n}^{\gb,\go}(\sfi(n)=i, \sfj(n)=j)
= \bar Z_{i}^{\gb,+}\times [2^{i-j-n-1}e^{-\gb(i+j-n)^2}] \times Z_{n-j}^{\gb,+},
\eeq
by the Markov property at times $i$ and $j$, and ``reversing time" on the interval $(j,n]$. The \emph{fully collapsed} state, corresponding to $\sfi(n)=0$ and $\sfj(n)=n$, yields the following lower bound:
\beq
\bar Z_{2n}^{\gb,\go} \ge \bar Z_{2n}^{\gb,\go}(\sfi(n)=0, \sfj(n)=n) =2^{1-2n}.
\eeq
(We remind the reader that the first step of the random walk is fixed to the value one). Therefore,
\beq
\bar \bP_{2n}^{\gb,\go}(\sfi(n)=i, \sfj(n)=j) \le
[2^{i-1}\bar Z_{i}^{\gb,+}]e^{-\gb(i+j-n)^2}[2^{n-j-1}\bar Z_{n-j}^{\gb,+}].
\eeq
By Lemma~\ref{lem:wsaw}, we have, since $\gb>\gb_0$,
\beq
2^{k-1}\bar Z_{k}^{\gb,+} \le C(\gb)^k, \qquad \text{with} \quad C(\gb):= 2e^{-\bar F_+(\gb)}\in(0,1),
\eeq
and (using~\eqref{eq:grand-can-pf-wsaw} below and the definition of $U(\gb)$ in~\eqref{eq:def-beta0})
\beq
\sum_{k\ge 0} 2^{k-1}\bar Z_{k}^{\gb,+} = \frac{1}{1-U(\gb)},
\eeq
which completes the proof.
\end{proof}
\begin{proof}[Proof of Lemma~\ref{lem:wsaw}]
The first part of the lemma is rather standard and follows from Fekete's lemma, once we notice that the sequence defined by
\beq
\bar Z_{n}^{\gb,+}(n\in\tau) = \bar Z_{n}^{\gb,+}(\gD S_n = 1) = (1/2) \bar Z_{n}^{\gb,+}, \qquad n\ge 1,
\eeq
is super-multiplicative. The second part of the lemma follows from
\beq
\ba
\bar Z_{n}^{\gb,+}(n\in\tau) &= \bE\Big[\exp\Big(-\gb \sum_{x\ge 1}\ell_n(x)^2\Big)\ind_{\{n\in\tau\}} \Big]\\
&= \sum_{k\ge1}\sum_{0=:t_0<t_1<\ldots<t_k=n} \prod_{1\le i \le k} e^{-\gb(t_i - t_{i-1})^2}2^{-(t_i-t_{i-1})}\\
\ea
\eeq
which implies that
\beq
\label{eq:grand-can-pf-wsaw}
\sum_{n\ge 1} e^{fn} \bar Z_{n}^{\gb,+}(n\in\tau) = \sum_{k\ge 1} \Big( \sum_{t\ge 1} e^{(f-\log 2)t-\gb t^2} \Big)^k, \qquad f\in\bbR.
\eeq
\end{proof}
\section{Existence and variational representation of the free energy}
\label{sec:var-rep-FE}
In this section we focus on the quenched free energy of the model, using one or the other definition, thanks to~\eqref{eq:equ-fe}. We start with the issue of existence and self-averaging in Section~\ref{sec:exist-sa}, along with a few preliminary properties of the free energy as a function of the inverse temperature. In Section~\ref{sec:LDP} we apply the so-called \emph{Large Deviation Principle (LDP) for words drawn in a random letter sequence} (see \cite{B08,BGDH10,BGDH10-add}) to derive a variational representation of the free energy. We will use this representation in Section~\ref{sec:high-temp-FE} and obtain new estimates on the free energy.
\subsection{Existence and self-averaging}
\label{sec:exist-sa}
Similarly to how we proved Lemma~\ref{lem:wsaw}, we define a {\it constrained} version of the quenched partition function by:
\beq
\label{eq:bar-pf-constr}
\bar Z^{\gb,\go}_{n,c} := \bar Z^{\gb,\go}_{n}(n\in \tau) = \bar Z^{\gb,\go}_{n}(\gD S_n = 1).
\eeq
Since the value of $\gD S_n$ does not affect the Hamiltonian, we readily get
\beq
\bar Z^{\gb,\go}_{n,c} = \frac12 \bar Z^{\gb,\go}_{n}.
\eeq
We prove the following:
\begin{theorem}
\label{thm:que-fe}
The quenched free energy exists as the {\it non-random} limit 
\beq
\bar F_{\rm que}(\gb) := \lim_{n\to\infty} \frac1n \log \bar Z^{\gb,\go}_n = \lim_{n\to\infty} \frac1n \log \bar Z^{\gb,\go}_{n,c} \le 0,
\eeq
which holds $\bbP$-almost surely and in $L^1(\bbP)$. Moreover,
\beq
\bar F_{\rm que}(\gb) = \sup_{n\ge 1}\frac1n  \bbE  \log \bar Z^{\gb,\go}_{n,c}.
\eeq
\end{theorem}
\begin{proof}[Proof of Theorem~\ref{thm:que-fe}]
Notice that $\bar Z^{\gb,\go}_{n+m,c} \ge \bar Z^{\gb,\go}_{n,c} \bar Z^{\gb,\theta^n\go}_{m,c}$, apply the logarithm function, and conclude via Kingman's superadditive ergodic theorem. 
\end{proof}
Let us list some properties of the quenched free energy as a function of the inverse temperature (see Figures~\ref{fig:QuenchedFE_N1000} and~\ref{fig:Averaged_QuenchedFE_N1000_Nsample100}):
\begin{proposition}
\label{pr:prop-fe}
We have the following:
\begin{enumerate}
\item\label{item:prop-fe1} $\gb\in [0,\infty) \mapsto \bar F_{\rm que}(\gb)$ is a convex non-increasing function.
\item\label{item:prop-fe2} $\bar F_{\rm que}(\gb) \ge (-\gb) \vee (-\log 2)$.
\end{enumerate}
\end{proposition}
We do not know whether the quenched free energy is differentiable as a function of the inverse temperature. If it were so, then by convexity the derivative would be the limit of the derivative of the finite-volume free energy, as $n\to\infty$ (see Item~\ref{item:DGH1} in Conjecture~\ref{conj:DGH}).
\begin{proof}[Proof of Proposition~\ref{pr:prop-fe}]
For Item~\ref{item:prop-fe1} it is enough to check that
\beq
\ba
\partial_\gb\Big(\frac1n \log \bar Z^{\gb,\go}_n\Big) &= - \bar \bE_n^{\gb,\go}\Big(\sum_{x\ge 1} \gO_n(x)^2\Big) \le 0,\\
\partial^2_\gb\Big(\frac1n \log \bar Z^{\gb,\go}_n\Big) &= \bar \bV_n^{\gb,\go}\Big(\sum_{x\ge 1} \gO_n(x)^2\Big) \ge 0.
\ea
\eeq
Let us deal with Item~\ref{item:prop-fe2}. By Jensen's inequality and the i.i.d.\ assumption on $\go$,
\beq
\bbE \log \bar Z^{\gb,\go}_n \ge - \gb \bbE \bE\Big( \sum_{1\le i,j\le n} \go_i \go_j \ind_{\{S_i = S_j\}}\Big) = -\gb n,
\eeq
which yields the first inequality. To obtain the second inequality, let us first assume for simplicity that $\gO_n >0$ and define, for every $1\le i \le \gO_n $,
\beq
\label{eq:random_times}
\gs_i = \sup\{1\le k \le n\colon \gO_k = i\}.
\eeq
By considering the one particular path which satisfies $\gD S_k = 1$ if and only if $k\in \gs=\{\gs_i,\ 1\le i \le \gO_n\}$, for $1\le k\le n$ (i.e. $\tau_i = \gs_i$ for $1\le i \le \gO_n $), we obtain
\beq
\frac1n \log \bar Z^{\gb,\go}_n \ge -\log 2 - \gb \frac{\gO_n}{n}.
\eeq
One may readily adapt the argument to the case $\gO_n \le0$ and obtain
\beq
\label{eq:quenched-low-temp-LB}
\frac1n \log \bar Z^{\gb,\go}_n \ge -\log 2 - \gb \frac{|\gO_n|}{n},
\eeq
leading to the following lower bound:
\beq
\frac1n \bbE \log \bar Z^{\gb,\go}_n \ge -\log 2 - \gb \frac{\bbE |\gO_n|}{n} = -\log 2 + o(1),
\eeq
which completes the proof of Item~\ref{item:prop-fe2}.
\end{proof}
\begin{figure}
  \centering
  \includegraphics[width=0.7\textwidth]{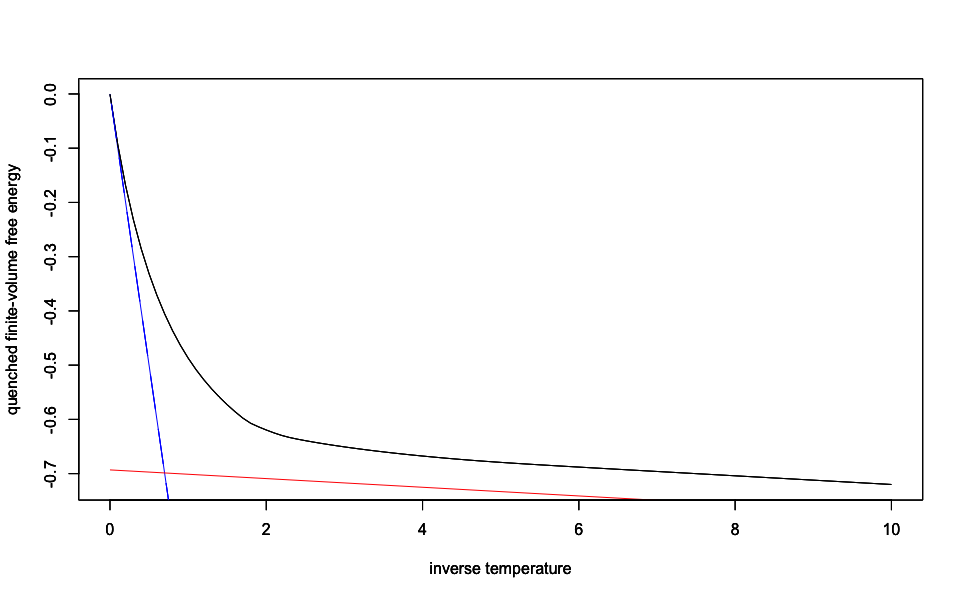}
  \caption{Quenched finite-volume free energy $(1/n) \log \bar Z^{\gb,\go}_n$ versus inverse temperature $\gb$ for centered $\pm 1$ i.id.\ charges $\go$ and polymer size $n=1000$ as defined in~\eqref{eq:bar-pf}. The high-temperature lower bound ($-\gb$) and the low-temperature lower bound from~\eqref{eq:quenched-low-temp-LB} respectively appear in blue and red.}
  \label{fig:QuenchedFE_N1000}
\end{figure}
\begin{figure}
  \centering
  \includegraphics[width=0.7\textwidth]{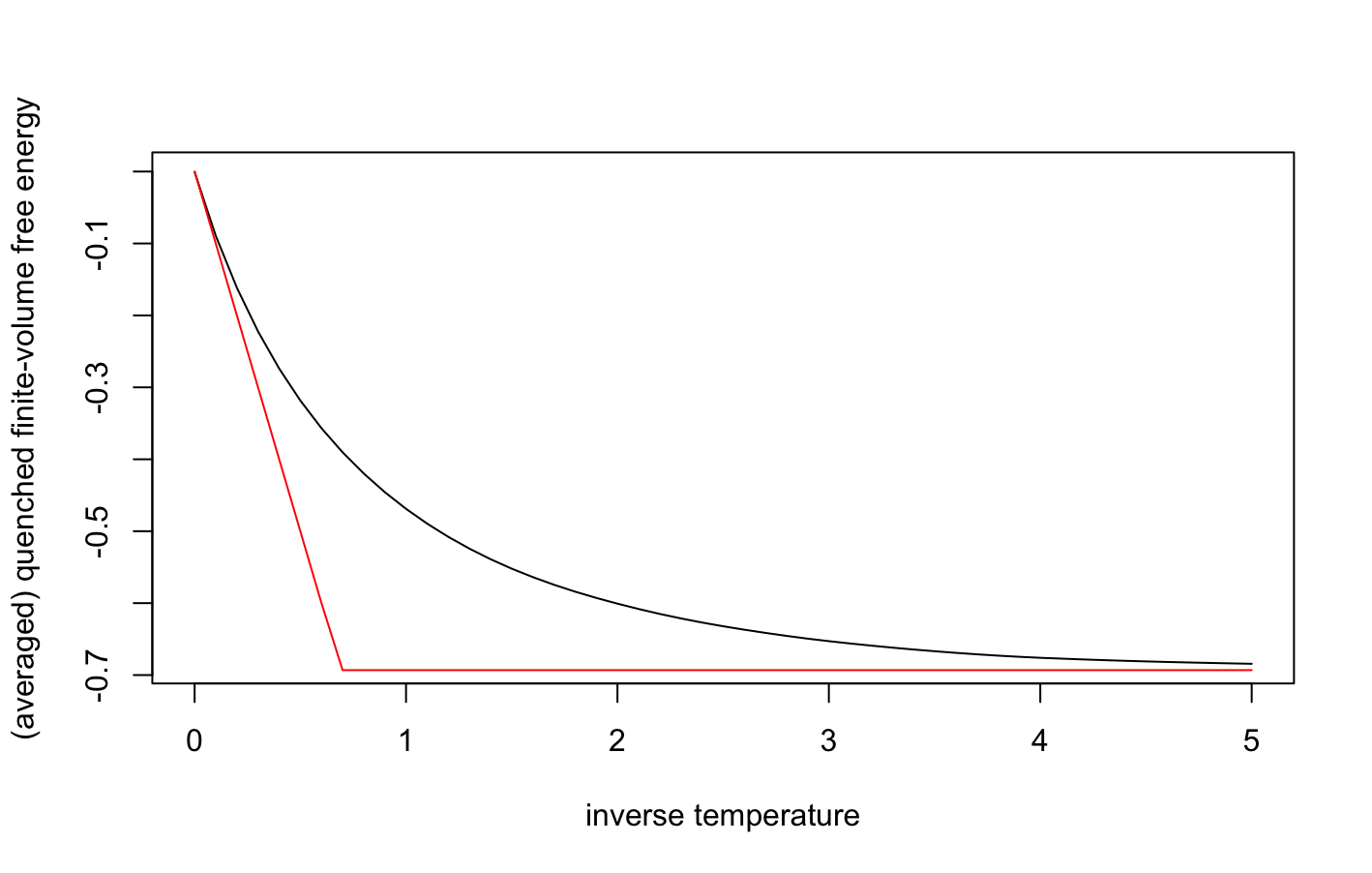}
  \caption{Averaged quenched finite-volume free energy $(1/n) \bbE  \log \bar Z^{\gb,\go}_n$ versus inverse temperature $\gb$ for centered $\pm 1$ i.id.\ charges $\go$ and polymer size $n=1000$ as defined in~\eqref{eq:bar-pf}. We have used $100$ samples to approximate the average over the charge distribution. The lower bound stated in Proposition~\ref{pr:prop-fe} is shown in red.}
  \label{fig:Averaged_QuenchedFE_N1000_Nsample100}
\end{figure}
\subsection{The LDP for words in a quenched random letter sequence}
\label{sec:LDP}
As we already observed in Section~\ref{sec:renewal_times}, the present model can be formulated as a statistical mechanics model built on a renewal sequence in the presence of quenched randomness, like the random pinning and copolymer models, see~\cite{dH09,Gia07,Gia11}. In these models, one can visualize the random charge sequence $(\go_i)_{i\ge 1}$ as a random \emph{letter sequence} from which the renewal process $(\tau_k)_{k\ge 0}$ cuts a \emph{word sequence} $(Y_k)_{k\ge 1}$, where $Y_k := (\go_i)_{\tau_{k-1}<i\le \tau_k}$ for every $k\ge 1$. Then, the energy or Hamiltonian function may be recast as an appropriate functional applied to the \emph{word sequence empirical measure}. The purpose of this section is to apply a large deviation principle (LDP) for the empirical measure of such words (when the letter sequence is quenched) in order to obtain a variational formula for the free energy. This LDP at the level of measures was originally introduced in~\cite{B08}, shortly later generalized to a broader class of renewal tails in~\cite{BGDH10,BGDH10-add} and successfully applied to the random pinning and copolymer models in~\cite{BdHO15,CdH2021,CdH13,dHO13,Mourrat12}.\\
\par To implement this idea, let us first define the grand canonical partition function:
\beq
\cZ_{\gb,\go}(f):= \sum_{n\ge 1} \bar Z_{n,c}^{\gb,\go} e^{-fn}, \qquad f\in\bbR.
\eeq
Note that we use the constrained version of the partition function defined in~\eqref{eq:bar-pf-constr}. By Theorem~\ref{thm:que-fe}, the free energy is the infimum of those values of $f$ which make the latter sum converge:
\beq
\label{eq:bar-f-que-gcpf}
\bar F_{\rm que}(\gb) = \inf\{f\in \bbR \colon \cZ_{\gb,\go}(f)< \infty\}.
\eeq
By~\eqref{eq:bar-pf2}, we may write
\beq
\label{eq:from-Z-f-to-W-k}
\ba
\cZ_{\gb,\go}(f) &= \sum_{n\ge 1} e^{-fn} \sum_{k\ge 1} \bE\Big[
\exp\Big(-\gb \sum_{1\le i\le k} \gO(\tau_{i-1}, \tau_i]^2\Big)
  \ind_{\{\tau_k = n\}}\Big]
\\
&= \sum_{k\ge 1} \bE\Big[\exp\Big(-\sum_{i=1}^k [\gb \gO(\tau_{i-1}, \tau_i]^2 + f(\tau_i- \tau_{i-1})]\Big) \Big]=:\sum_{k\ge 1} W_k^{\gb,\go}.
\ea
\eeq
Using the Large Deviation Principle for words cut out of a quenched random sequence of letters and Varadhan's lemma, see~\cite[Theorem 1 and Corollary 1]{B08}, we obtain the following $\bbP$-a.s.\ limit:
\beq
\label{eq:from-W-k-to-J-f}
\lim_{k\to \infty} \frac 1k \log W_k^{\gb,\go} = -J(\gb,f), \qquad f>-\log 2,
\eeq
where
\beq
\label{eq:var-repr-J}
J(\gb, f):= \inf_{Q\in \cR_0} \Big\{\int [\gb(y_{1,1} + \ldots +y_{1,\ell(y_1)})^2 + f \ell(y_1)] Q(\dd y) + H(Q| Q_0)\Big\}
\eeq
and with the following notation:
\begin{itemize}
\item $Q$ is a probability distribution on sequences of \emph{words}. Words are elements of $\cup_{\ell\ge 1} \bbR^\ell$, and $\ell(y_1)$ denotes the length of $y_1= (y_{1,1}, \ldots, y_{1,\ell(y_1)})$, that is the first word in the sequence $y = (y_i)_{i\ge 1}$.
\item $Q_0$ is the probability distribution for i.i.d.\ words with marginal
\beq
\label{eq:def-Q0}
Q_0(\dd y_1) = \sum_{\ell\ge 1} 2^{-\ell} \prod_{1\le i \le \ell} \bbP(\dd y_{1,i}),
\eeq
that is the one induced by the charge/letter sequence law $\bbP$ (Section~\ref{sec:origin-conv}) and the geometrically distributed inter-arrival times $(\tau_k - \tau_{k-1})_{k\ge 1}$ (Section~\ref{sec:renewal_times}).
\item $H(Q|Q_0)$ is the \emph{specific relative entropy} of $Q$ w.r.t.\ $Q_0$, defined as
\beq
\label{eq:spec-rel-ent}
H(Q|Q_0) = \lim_{n\to \infty} \frac 1n h(Q[(Y_1, \ldots, Y_n)\in \cdot]\, |\,  Q_0[(Y_1, \ldots, Y_n)\in \cdot]),
\eeq
where $h(\cdot | \cdot)$ is the usual relative entropy of one probability measure w.r.t.\ another and the limit is non-decreasing.
\item $\cR_0$ denotes the set of shift-invariant distributions on word sequences which are \emph{compatible with $\bbP$}, meaning that when one concatenates a typical word sequence into a letter sequence, the empirical distribution of the latter (at the level of processes) converges weakly to $\bbP$, see~\cite[Equation (8)]{B08} for a formal definition.
\end{itemize}
Note that $J(\gb,f)$ is non-decreasing in both its variables. The fact that~\eqref{eq:from-W-k-to-J-f} is not stated for all the possible values of $f$ comes from the lack of boundedness of the function to which we apply Varadhan's lemma. As in~\cite{B08} (see the remark around Eq. (19) therein) we  use instead an exponential tightness property~\cite[Condition (4.3.3) in Theorem 4.3.1]{DemboZei10:book}, hence the restriction to $f>-\log 2$. Indeed, in that case one can find some $\gga>1$ such that, for every $\go$,
\beq
\bE\Big(e^{-\gga(\gb \gO(0,\tau_1]^2 + f\tau_1)}\Big) \le \bE(e^{-\gga f \tau_1}) = \sum_{n\ge 1} \big(\tfrac12 e^{-\gga f} \big)^n < + \infty.
\eeq
The restriction to $f>-\log 2$ is  actually harmless since we know from a separate argument that $\bar F_{\rm que}(\gb)\ge -\log 2$ for every $\gb>0$, see Proposition~\ref{pr:prop-fe}. 
Combining~\eqref{eq:bar-f-que-gcpf}, \eqref{eq:from-Z-f-to-W-k} and~\eqref{eq:from-W-k-to-J-f}, we finally obtain the following:
\begin{theorem}
\label{thm:var-repr-fe}
For every $\gb>0$,
\beq
\bar F_{\rm que}(\gb) = \inf\{f\in \bbR \colon J(\gb,f) > 0\}.
\eeq
\end{theorem}
Let us comment on a few elementary bounds obtained from this formula. By plugging $Q = Q_0$ into~\eqref{eq:var-repr-J} we obtain that $J(\gb,f) \le 2(\gb +f)$, from which we retrieve the elementary Jensen bound $\bar F_{\rm que}(\gb) \ge -\gb$. One can also try to restrict the infimum in~\eqref{eq:var-repr-J} to those $Q$'s under which words are i.i.d.\ with marginal of the form
\beq
Q(\dd y_1) = \sum_{\ell \ge 1} K(\ell) \prod_{1\le i \le \ell} \bbP(\dd y_{1,i}),
\eeq 
where $K$ ranges over all probability distributions on the set of positive integers. We obtain thereby
\beq
J(\gb,f) \le \inf_K\{(f+\log 2 + \gb)m(K) - h(K)\},
\eeq
where $m(K) := \sum_{\ell \ge 1} \ell K(\ell)$ and $h(K) := - \sum_{\ell \ge 1} K(\ell) \log K(\ell)$ are the mean and entropy of $K$, respectively.
A simple argument (using Lagrange multipliers) shows that the optimal $K$ follows a geometric distribution, i.e.\ $K(\ell) = \gt^{\ell-1}(1-\gt)$ for some $\gt\in(0,1)$. Therefore,
\beq
J(\gb,f) \le \inf_{\gt\in(0,1)} \Big\{\frac{f+\log 2 + \gb}{1-\gt} + \Big(\frac{\gt}{1-\gt}\Big)\log \gt + \log(1-\gt) \Big\}.
\eeq
The latter is optimal at $\gt_0 := \exp(-[f+\log 2 + \gb])$, provided $\gt_0\in(0,1)$, leading in that case to $J(\gb,f) \le \log(\exp(f+\log 2 + \gb)-1)$, and ultimately to $\bar F_{\rm que}(\gb) \ge - \gb$. Unfortunately, this is not any better than the Jensen lower bound. The reason is that this strategy does not take under consideration the charge sequence. In Section~\ref{sec:high-temp-FE}, we shall implement a more refined strategy where each monomer looks at the charge of the following one to decide its state (folded or stretched) and improve thereby our lower bound on the free energy.
\section{High-temperature estimate on the free energy}
\label{sec:high-temp-FE}
This section is primarily focussed on high-temperature estimates on the quenched free energy of the charged polymer, even if some results that we derive {\it en route} (namely Theorem~\ref{thm:que-LB} and Proposition~\ref{eq:annealed-bound}) are valid at all temperatures.
\subsection{A lower bound}
The starting point of our lower bound is Theorem~\ref{thm:var-repr-fe}. By testing an appropriate probability distribution on word sequences in the variational formula~\eqref{eq:var-repr-J}, we obtain an upper bound on the rate function $J(\gb,f)$, which in turn yields a lower bound on the free energy. The distribution $Q\in \cR_0$ which we test is a tilt of the $\geo(1/2)$ distribution which looks at pairs of consecutive charges to decide on where to fold the polymer chain, see~\eqref{eq:cond-law-Qu0} below. In view of Proposition~\ref{pr:DGH}, it is plausible that a better strategy would look at \emph{all} pairs of charges, with a decreasing influence for charges that are far apart, but the implementation of such strategy seems technically much more demanding.
\begin{theorem}
\label{thm:que-LB}
For every $\gb>0$,
\beq
\bar F_{\rm que}(\gb) \ge -\gb + \sup_{0\le u \le \log 2} \Big[ \gb \frac{\sinh(u)}{1+\frac{\sinh(u)}{2}} - \frac{1}{2} u \sinh(u) - \eta(u) \Big] ,
\eeq
where
\beq
\label{eq:def-eta-u}
\eta(u) := \frac14[\log(2-e^u)(2-e^u)+\log(2-e^{-u})(2-e^{-u})].
\eeq
\end{theorem}
We readily deduce thereof the following high temperature lower bound:
\begin{corollary} As $\gb \to 0$,
\label{eq:que-high-temp}
\beq
\bar F_{\rm que}(\gb) \ge -\gb + \frac {\gb^2}{2} [1+o(1)].
\eeq
\end{corollary}
\begin{remark}
\label{rmk:naive-exp}
Expanding the finite-volume free energy at high temperature (see Proposition~\ref{pr:finite-fe-high-temp} for the precise assumptions) and (blindly) interchanging the $\gb\to0$ and $n\to\infty$ limits leads to the prediction:
\beq
\bar F_{\rm que}(\gb) \stackrel{(?)}{=} -\gb + \frac {4\gb^2}3  [1+o(1)].
\eeq
However, such naive expansions do not always give the correct values for high-temperatures limits, see e.g.~\cite{BCPSZ14} in the context of pinning and copolymer models.
\end{remark}
\begin{proof}[Proof of Theorem~\ref{thm:que-LB}]
For every $u\in[0,\log 2]$, we introduce a law on word sequences $Y=(Y_i)_{i\ge 1}$, denoted by $Q_u$, such that (i) the letter sequence distribution equals the original charge sequence distribution, i.e.\ $Q_u(\go = \kappa(Y) \in \cdot\,) = \bbP(\go\in \cdot\,)$, where $\kappa$ is the \emph{word-to-letter sequence concatenation}, and (ii) the length of the first word, denoted by $L$, is distributed as
\beq
\label{eq:cond-law-Qu0}
Q_u(L=\ell\, |\, \go) = \prod_{1\le i < \ell} (1-\tfrac12 e^{u \go_i \go_{i+1}}) \tfrac12 e^{u \go_\ell \go_{\ell+1}}, \qquad \ell\ge 1.
\eeq
In other words, $L =\inf\{i\ge 1\colon \gt_i = 1\}$ where $(\theta_i)_{i\ge 1}$ is a sequence of Bernoulli random variables that are independent (conditionally to $\go$) and such that
\beq
\label{eq:cond-law-Qu}
Q_u(\gt_i = 1\, |\, \go) = \tfrac12 e^{u\go_i \go_{i+1}}, \qquad i\ge 1.
\eeq
A similar formula is assumed to hold for the length of the next words, simply by appropriately shifting the letter sequence $\go$ on the r.h.s.\ of~\eqref{eq:cond-law-Qu0}. Note that the random variables $(\go_i\go_{i+1})_{i\ge 1}$ are actually independent and uniformly distributed on $\{-1,1\}$. Following this remark and averaging~\eqref{eq:cond-law-Qu0} over $\go$ yields
\beq
Q_u(L = \ell) = \Big(1- \frac{\cosh(u)}{2}\Big)^{\ell -1} \frac{\cosh(u)}{2}, \qquad \ell \ge 1,
\eeq
that is the geometric distribution with expectation $2/\cosh(u)$. The reader may check that under $Q_u$ the word sequence is a Markov chain with transitions
\beq
\mathfrak{Q}_u((\go_i)_{1\le i \le \ell_1}, (\go_i)_{\ell_1 < i \le \ell_2}) := 2^{\ell_1- \ell_2} (\tfrac 12 e^{u \go_{\ell_1}\go_{\ell_1+1}}) \prod_{\ell_1 < i < \ell_2}(1 - \tfrac12 e^{u\go_i \go_{i+1}})
\eeq
(the distribution of a given word depends on the last letter of the previous word) and started from the invariant word (probability) distribution 
\beq
\pi_u(\go_1, \ldots, \go_{\ell}) := 2^{-(\ell+1)}\cosh(u)\prod_{1\le i< \ell} (1- \tfrac12 e^{u\go_i \go_{i+1}}).
\eeq
Moreover, words are i.i.d.\ if and only if $u=0$, in which case one retrieves the law $Q_0$ defined in~\eqref{eq:def-Q0}. From all these observations we deduce that $Q_u\in \cR_0$, that is the restriction set of shift-invariant word sequence distributions which are compatible with $\bbP$, see the definition below~\eqref{eq:var-repr-J}.\\
\par  We may now plug $Q = Q_u$ into~\eqref{eq:var-repr-J}, so that we obtain
\beq
J(\gb,f)\le \gb \int (\go_1 + \ldots + \go_L)^2 \dd Q_u + f \int L\, \dd Q_u + H(Q_u | Q_0).
\eeq
We begin with computing the (specific) relative entropy. By~\eqref{eq:spec-rel-ent} and the ergodic theorem, we have
\beq
\ba
H(Q_u | Q_0) &= \int \log \mathfrak{Q}_u(Y_1,Y_2) \dd Q_u - \int \log \pi_0(Y_1) \dd Q_u\\
&= \int \Big[\sum_{1\le i < L} \log(2-e^{u\go_i\go_{i+1}}) + u \go_L\go_{L+1} \Big] \dd Q_u.
\ea
\eeq
 With the help of Lemma~\ref{lem:comp-prod-go-Qu} below, we find that
\beq
H(Q_u | Q_0) = \frac{2}{\cosh(u)}\Big(\eta(u) + \frac12 u \sinh(u) \Big).
\eeq
Using once again Lemma~\ref{lem:comp-prod-go-Qu}, we may compute the contribution of this strategy to the energy term:
\beq
\ba
\int (\go_1 + \ldots + \go_L)^2 \dd Q_u &= 
\int L \dd Q_u + 
2\sum_{1\le i < j}\int \go_i \go_j \ind_{\{L\ge j\}}  \dd Q_u\\
&= \frac{2}{\cosh(u)}\Big(1 - \frac{\sinh(u)}{1+\frac{\sinh(u)}{2}} \Big).
\ea
\eeq
Summing up the different terms, we finally obtain
\beq
J(\gb, f) \le \frac{2}{\cosh(u)} \Big[ \gb \Big( 1 - \frac{\sinh(u)}{1+\frac{\sinh(u)}{2}}\Big) + \frac{1}{2} u \sinh(u) + \eta(u) + f \Big].
\eeq
Hence, by Theorem~\ref{thm:var-repr-fe},
\beq
\bar F_{\rm que}(\gb) \ge \gb \Big( \frac{\sinh(u)}{1+\frac{\sinh(u)}{2}}-1\Big) - \frac{1}{2} u \sinh(u) - \eta(u) ,
\eeq
and the proof is complete.
\end{proof}

\begin{proof}[Proof of Corollary~\ref{eq:que-high-temp}]
Let us pick $u=C\gb$ in Theorem~\ref{thm:que-LB}, with $C>0$ to be determined later. Noting that, as $u\to 0$,
\beq
\frac{\sinh(u)}{1+\frac{\sinh(u)}{2}} \sim u, \qquad u\sinh(u) \sim u^2,
\qquad \eta(u) = o(u^2),
\eeq
we obtain, as $\gb\to 0$,
\beq
\bar F_{\rm que}(\gb) \ge -\gb + \gb^2(C - C^2/2) + o(\gb^2),
\eeq
which we optimise by setting $C=1$.
\end{proof}
\begin{lemma}
\label{lem:comp-prod-go-Qu}
For every $\Phi\colon \{-1,1\}\mapsto \bbR$ and $u\in [0, \log 2]$,
\beq
\label{lem:comp-prod-go-QuA}
\ba
\int \Phi(\go_i\go_{i+1}) \ind_{\{L>i\}} \dd Q_u &= \Big(1 - \frac{\cosh(u)}{2}\Big)^{i-1} \, \frac{\Phi(-1)(2-e^{-u})+\Phi(1)(2-e^{u})}{4},\\
\int \Phi(\go_i\go_{i+1}) \ind_{\{L=i\}} \dd Q_u &= \Big(1 - \frac{\cosh(u)}{2}\Big)^{i-1} \, \frac{\Phi(-1)e^{-u}+\Phi(1)e^u}{4},
\ea
\eeq
and for every $1\le i <j$,
\beq
\label{lem:comp-prod-go-QuB}
\int \go_i\go_j \ind_{\{L\ge j\}} \dd Q_u = \Big(1 - \frac{\cosh(u)}{2}\Big)^{i-1}\Big(- \frac{\sinh(u)}{2}\Big)^{j-i}.
\eeq
\end{lemma}

\begin{proof}[Proof of Lemma~\ref{lem:comp-prod-go-Qu}]
Let us start with~\eqref{lem:comp-prod-go-QuA}. By conditioning on $\go$ and using~\eqref{eq:cond-law-Qu0},
\beq
\ba
\int \Phi(\go_i\go_{i+1}) \ind_{\{L>i\}} \dd Q_u &= \int \Phi(\go_i\go_{i+1}) \prod_{1\le j \le i} (1-\tfrac12 e^{u \go_j \go_{j+1}}) \dd \bbP,\\
\int \Phi(\go_i\go_{i+1}) \ind_{\{L=i\}} \dd Q_u &= \int \Phi(\go_i\go_{i+1})  \prod_{1\le j < i} (1-\tfrac12 e^{u \go_j \go_{j+1}}) \tfrac12 e^{u \go_i \go_{i+1}} \dd \bbP.
\ea
\eeq
We conclude by observing that the random variables $(\go_i\go_{i+1})_{i\ge 1}$ are independent and uniformly distributed in $\{-1,1\}$. The equality in~\eqref{lem:comp-prod-go-QuB} is proven in a similar fashion, once we notice that $\go_i \go_j  = \prod_{i\le k < j} \go_k \go_{k+1}$.
\end{proof}
\subsection{An upper bound}
In this section we derive several upper bounds on the quenched free energy through \emph{annealing}. We start with the following: 
\begin{proposition}
\label{pr:annealed-bound}
For every $\gb\ge 0$,
\beq
\label{eq:annealed-bound}
\bar F_{\rm que}(\gb)\le \bar F_{\rm ann}(\gb) := - \sup\Big\{f\ge 0\colon \sum_{\ell\ge 1} (e^f/2)^{\ell} \bbE(e^{-\gb \gO_\ell^2}) < 1\Big\}\in [-\log 2, 0].
\eeq
\end{proposition}
\begin{proof}[Proof of Proposition~\ref{pr:annealed-bound}]
For ease of notation, let us set
\beq
g_\gb(\ell) := \bbE( e^{-\gb \gO_\ell^2} ) = \bbE( e^{-\gb (\go_1 + \ldots + \go_\ell)^2} ),
\qquad \ell\ge 1,
\eeq
with the convention $g_\gb(0) = 1$.
Using the i.i.d.\ assumption on the charges, the annealed partition function equals
\beq
\bbE \bar Z_n^{\gb,\go} = \bE\Big[ \prod_{x\in \bbN} g_\gb(\ell_n(x))\Big],  
\qquad
\textrm{with }
\ell_n(x) := \card\{1\le k\le n\colon S_k = x\}.
\eeq
Using the geometric distribution of the time spent on each visited vertex, we readily obtain
\beq
\label{eq:ann-grand-can-pf}
\bbE \bar Z_n^{\gb,\go} = 2\sum_{k=1}^n \sumtwo{\ell_1, \ldots, \ell_k \ge 1}{\ell_1 + \ldots + \ell_k = n} \prod_{i=1}^k g_\beta(\ell_i) 2^{-\ell_i},
\eeq
from which we deduce the annealed grand canonical partition function
\beq
\label{eq:ann-grand-can-pf2}
\sum_{n\ge 1} \bbE \bar Z_n^{\gb,\go} e^{fn} = 2\sum_{k \ge 1} 
\Big[ \sum_{\ell \ge 1} (e^f/2)^{\ell} g_\gb(\ell) \Big]^k.
\eeq
The extra factor $2$ in front of the sums in~\eqref{eq:ann-grand-can-pf} and~\eqref{eq:ann-grand-can-pf2} comes from the fact that $S_1=1$. Thus, the annealed free energy, which could alternatively be defined as 
\beq
\label{eq:ann-fe}
\bar F_{\rm ann}(\gb) = \limsup_{n\to \infty} \frac 1n \log \bbE \bar Z_n^{\gb,\go},
\eeq
has the variational representation displayed in~\eqref{eq:annealed-bound}. The first inequality therein follows from Theorem~\ref{thm:que-fe} and the standard Jensen bound:
\beq
\bbE \log \bar Z_n^{\gb,\go} \le \log \bbE \bar Z_n^{\gb,\go}.
\eeq
The fact that the annealed free energy is non-positive clearly follows from~\eqref{eq:ann-fe}. The fact that it is bounded from below by $(-\log 2)$ follows by checking that
\beq
\sum_{\ell \ge 1} (e^
{\log 2}/2)^{\ell} g_\gb(\ell) = \sum_{\ell \ge 1} g_\gb(\ell) \ge \sum_{\ell \in 2\bbN} \bbP(\gO_\ell = 0) = +\infty,
\eeq
since $\bbP(\gO_\ell = 0) \sim \cst\, \ell^{-1/2}$ as $\ell\to\infty$ along even integers.
\end{proof}
We first use~Proposition~\ref{pr:annealed-bound} to derive a high-temperature upper bound.
\begin{proposition}[Binary charges]
\label{pr:annealed-bound-hightemp}
As $\gb \to 0$,
\beq
\bar F_{\rm que}(\gb) \le \bar F_{\rm ann}(\gb) \le -\gb + 2\gb^2 + o(\gb^2).
\eeq
\end{proposition}
\begin{proposition}[standard Gaussian charges]
\label{pr:annealed-bound-hightempG}
As $\gb \to 0$,
\beq
\bar F_{\rm que}(\gb) \le \bar F_{\rm ann}(\gb) \le -\gb + 3\gb^2 + o(\gb^2).
\eeq
\end{proposition}
\begin{proof}[Proof of Proposition~\ref{pr:annealed-bound-hightemp}]
Recall the expression of the annealed free energy in~\eqref{eq:annealed-bound}. Replacing the $(e^{f}/2)$ therein by $z < 1$, using the inequality $e^{-y} \le 1-y + \tfrac12 y ^2$ for every $y\ge 0$, and using Lemma~\ref{lem:fourth-mom}, we have
\beq
\sum_{\ell\ge 1} z^{\ell} \bbE(e^{-\gb \gO_\ell^2}) \le \sum_{\ell\ge 1} z^{\ell}
[1 - (\gb+\gb^2)\ell + \tfrac32 \gb^2 \ell^2].
\eeq
Using Lemma~\ref{lem:ann-bound} with $\sfc_1(\gb) = -\gb - \gb^2$ and $\sfc_2(\gb) = \tfrac32 \gb^2$, we obtain
\beq
\label{eq:suff-cond}
2z^3 + ({\sf u}_1(\gb) - 5)z^2 + ({\sf u}_2(\gb)+4)z - 1 < 0 \Rightarrow \sum_{\ell\ge 1} z^{\ell} \bbE(e^{-\gb \gO_\ell^2}) < 1,
\eeq
where
\beq
\sfu_1(\gb) = \gb + \tfrac52 \gb^2, \qquad
\sfu_2(\gb) = -\gb + \tfrac12 \gb^2.
\eeq
Using Lemma~\ref{lem:deg3-polynomial}, we see that the condition on the l.h.s.\ of~\eqref{eq:suff-cond} is satisfied if
\beq
z < z(\gb) := z_0(\sfu_1(\gb), \sfu_2(\gb)) = \tfrac12 + \tfrac12 \gb - \tfrac34 \gb^2  + o(\gb^2).
\eeq
Therefore,
\beq
\bar F_{\rm ann}(\gb) \le - \log(2z(\gb)) = - \gb + 2\gb^2 + o(\gb^2).
\eeq
\end{proof}
\begin{proof}[Proof of Proposition~\ref{pr:annealed-bound-hightempG}]
Same as Proposition~\ref{pr:annealed-bound-hightemp} with $\bbE(\gO_\ell^4) = 3\ell^2$, hence $\sfc_1(\gb)= -\gb$ and $\sfc_2(\gb) = \tfrac32 \gb^2$.
\end{proof}
\section{Low-temperature estimate on the free energy}
\label{sec:low-temp-FE}
In this section we use Proposition~\ref{pr:annealed-bound} to derive \emph{low-temperature} upper bounds on the quenched free energy, namely:
\begin{proposition}[Gaussian case]
\label{lem:ann-bound-low-temp}
As $\gb \to \infty$,
\beq
\bar F_{\rm que}(\gb)\le \bar F_{\rm ann}(\gb) = - \log 2 + \frac{\pi}{2\gb}(1+o(1)),
\eeq
\end{proposition}
\begin{proof}[Proof of Proposition~\ref{lem:ann-bound-low-temp}]
The inequality has already been proved so we focus on the low-temperature expansion.
Recall the expression of the annealed free energy in~\eqref{eq:annealed-bound}.
In the case of standard Gaussian charges, we can explicitely compute
\beq
\bbE(e^{-\gb \gO_\ell^2}) = \frac{1}{\sqrt{1+2\gb \ell}}.
\eeq
Considering $f(\gb) := \log 2 - C/\gb$, with $C>0$, we obtain
\beq
\sum_{\ell\ge 1} (e^{f(\gb)}/2)^{\ell} \bbE(e^{-\gb \gO_\ell^2}) 
= \sum_{\ell\ge 1} \frac{e^{-C\ell / \gb}}{\sqrt{1+2\gb \ell}}.
\eeq
Dropping one in the square root above yields the following upper bound:
\beq
\sum_{\ell\ge 1} (e^{f(\gb)}/2)^{\ell} \bbE(e^{-\gb \gO_\ell^2}) \le 
\sum_{\ell\ge 1} \frac{e^{-C\ell / \gb}}{\sqrt{2\gb \ell}} = \frac{1}{\gb} \sum_{\ell\ge 1} \frac{e^{-C\ell / \gb}}{\sqrt{2\ell/\gb}},
\eeq
which converges in the large $\gb$ limit, by a Riemann approximation, to
\beq
\int_0^{\infty} (2x)^{-1/2}e^{-Cx}\dd x = \frac{1}{\sqrt{C}} \int_0^{\infty} (2x)^{-1/2}e^{-x}\dd x \stackrel{(y^2=2x)}{=} \sqrt{\frac{\pi}{2C}}.
\eeq
The upper bound that we have just used can be complemented by the following lower bound:
\beq
\frac{e^{-C\ell / \gb}}{\sqrt{1+2\gb \ell}} = \frac{e^{-C\ell / \gb}}{\sqrt{2\gb \ell}} \times \frac{1}{\sqrt{1 + (2\gb\ell)^{-1}}} \ge \frac{e^{-C\ell / \gb}}{\sqrt{2\gb \ell}} \times \frac{1}{\sqrt{1 + (2\gb)^{-1}}}, \qquad \ell \ge 1,
\eeq
which allows to complete the proof.
\end{proof}
\begin{remark}[Binary case]
The annealed upper bound obtained in the case of binary charges is quite rough. Indeed, one can show that
\beq
\label{eq:low-temp-ann-limit}
\lim_{\gb\to\infty} \bar F_{\rm ann}(\gb) = -\frac{\log 3}2 > -\log 2.
\eeq
Existence follows from monotonicity. In order to obtain the value of the limit, we first observe that
\beq
\label{eq:low-temp-ann-limit2}
0\le \sum_{\ell \ge 1} z^\ell \bbE(e^{-\gb \gO_\ell^2}) - \sum_{\ell \ge 1} z^{2\ell} \bbP(\gO_{2\ell} = 0) \le \frac{e^{-\gb}z}{1-z}, \qquad 0\le z < 1.
\eeq
Indeed, by splitting the leftmost sum according to the parity of $\ell$, we obtain:
\beq
\ba
0\le &\sumtwo{\ell \ge 1}{\ell \in 2\bbZ} z^\ell \bbE(e^{-\gb \gO_\ell^2}) - \sumtwo{\ell \ge 1}{\ell \in 2\bbZ} z^\ell \bbP(\gO_\ell=0) \le e^{-\gb} \sumtwo{\ell \ge 1}{\ell \in 2\bbZ} z^\ell,\\
0\le &\sumtwo{\ell \ge 1}{\ell \in 2\bbZ+1} z^\ell \bbE(e^{-\gb \gO_\ell^2}) \le e^{-\gb} \sumtwo{\ell \ge 1}{\ell \in 2\bbZ+1} z^\ell,
\ea
\eeq
from which~\eqref{eq:low-temp-ann-limit2} readily follows. The rightmost sum in~\eqref{eq:low-temp-ann-limit2} can be computed explicitly, see e.g.~\cite[Chapter I, E2]{Spitzer76-book}
\beq
\sum_{\ell \ge 1} z^{2\ell} \bbP(\gO_{2\ell} = 0) = \frac{1}{\sqrt{1-z^2}} - 1, \qquad |z|<1.
\eeq
Combining the above expression with~\eqref{eq:low-temp-ann-limit2} and~\eqref{eq:annealed-bound} gives~\eqref{eq:low-temp-ann-limit}. 
The main contribution to the annealed partition function, in the large $\beta$ limit, is thus given by charge distributions such that $\gO_n(x) = 0$ for every $x\in \{S_1,\ldots, S_n\}$. This indicates that any attempt to improve the annealed bound by means of \emph{constrained annealing} with a linear potential, that is 
\beq
\bbE \log \bar Z_n^{\gb,\go} \le \log \bbE [\bar Z_n^{\gb,\go} \exp(\gl \gO_n)], \qquad \gl\in\bbR,
\eeq
(also known as first order Morita approximation) is bound to fail. See e.g.\ ~\cite{CG2005-ecp} for a reference on constrained annealing in the context of pinning 
\end{remark}
\section{Conclusion and perspectives}
For the \emph{directed} charged polymer with quenched \emph{centered} charges, we have reviewed and provided detailed proofs of past results concerning the freezing transition predicted by Derrida, Griffiths and Higgs. We showed that the quenched free energy enjoys a variational representation based on a quenched (or conditional) large deviation principle, much like pinning and copolymer models. Lower bounds are derived from this variational formula, while upper bounds are obtained through annealing. It is however not clear to us whether the freezing transition can be read from the free energy. Nevertheless, we hope that the efforts to obtain sharper estimates will lead to a better understanding of the model.\\

\par Let us finally close the paper with a (non-exhaustive) list of unsettled issues for the directed model:
\begin{enumerate}
\item Settle Items~\ref{item:DGH1} and~\ref{item:DGH2} in Conjecture~\ref{conj:DGH}.
\item Give a rigorous proof for the existence of a low temperature regime in the case of random centered charges, as in Item~\ref{item:DGH3} in Conjecture~\ref{conj:DGH}. 
\item What is the critical temperature for the di-block polymer from~\eqref{eq:two-block-charges}? The best estimate so far is $\log(\frac{1+\sqrt{5}}{2}) \le \gb_c^{\rm di-block} \le 2\log 2$. Approximate numerical values for these bounds are $0.481$ and $1.386$.
\item Is the inequality $\bar F_{\rm que}(\gb) \ge -\log 2$ strict? Is it valid beyond the binary charge distribution? The strategy used in the proof of Item~\ref{item:prop-fe2} in Proposition~\ref{pr:prop-fe} could be adapted to other charge distributions (e.g.\ Gaussian charges) by folding the chain along the sequence of stopping times defined by $\gs_0=0$ and
\beq
\gs_i = \min\{k> \gs_{i-1}\colon \gO_k > \gO_{\gs_{i-1}}\}, \qquad i\in\bbN, 
\eeq
instead of~\eqref{eq:random_times} (assuming w.l.o.g.\ that $\gO_n >0$). However, the fact that we must stop at $i_n := \min\{i\in\bbN\colon \gO_{\gs_i}\ge \gO_n\}$, that is not a stopping time, would probably lead to cumbersome technicalities.
\item Is there a minimizer for the variational problem in~\eqref{eq:var-repr-J}? If so, is it unique? 
\item Is the quenched free energy analytic as a function of the inverse temperature?
\item Find the value of the constant $\mathsf{C}$ in the high temperature expansion
\beq
\bar F_{\rm que}(\gb) = -\gb + \mathsf{C}\gb^2 + o(\gb^2), \qquad \gb\to 0.
\eeq
Is it universal? Our best estimate so far is $1/2 \le \mathsf{C} \le 2$ for binary charges and $1/2 \le \mathsf{C} \le 3$ for Gaussian charges.
\item Can the lower bound in Theorem~\ref{thm:que-LB} be improved by considering a more subtle strategy (looking beyond consecutive monomers)?
\item Can we prove that $\bar F_{\rm que}(\gb) \sim -\log 2$ as $\gb\to\infty$ in the case of binary charges? An upper bound is missing.
\end{enumerate}
\appendix

\section{Technical estimates}

\begin{lemma}
\label{lem:fourth-mom}
In the case of $\{-1,1\}$-valued centered i.i.d.\ charges, for every $\ell \ge 1$,
\beq
\bbE(\gO_\ell^4) = 3 \ell^2 - 2\ell.
\eeq
\end{lemma}
\begin{proof}[Proof of Lemma~\ref{lem:fourth-mom}] We have
\beq
\bbE(\gO_\ell^4) = \bbE\Big(\Big[\ell + 2 \sum_{1\le i<j\le \ell} \go_i \go_j\Big]^2\Big) = \ell^2 + 4 \bbE\Big(\sumtwo{1\le i<j\le \ell}{1\le u<v\le \ell} \go_i \go_j \go_u \go_v\Big),
\eeq
and the result follows by noticing that only the terms for which $i=u$ and $j=v$ give a nonzero (and actually unit) contribution to the last expectation.
\end{proof}

\begin{lemma}
\label{lem:ann-bound}
Let $|z|<1$. Then,
\beq
\label{eq:ann-bound}
\sum_{\ell \ge 1} z^\ell (1  + {\sf c}_1\ell + {\sf c}_2\ell^2) < 1
\Leftrightarrow 
2z^3 + ({\sf u}_1 - 5)z^2 + ({\sf u}_2+4)z - 1 < 0,
\eeq
where
\beq
{\sf u}_1 = {\sf c}_2 - {\sf c}_1, \qquad
{\sf u}_2 = {\sf c}_1 + {\sf c}_2.
\eeq
\end{lemma}
\begin{proof}[Proof of Lemma~\ref{lem:ann-bound}]
The sum in the l.h.s.\ of~\eqref{eq:ann-bound} equals
\beq
\sum_{\ell \ge 1} z^\ell   + [\sfc_1 + \sfc_2] \sum_{\ell \ge 1} \ell z^\ell + {\sf c}_2\sum_{\ell \ge 1} \ell(\ell-1) z^\ell 
=\frac{z}{1-z} + \frac{[\sfc_1 + \sfc_2]z}{(1-z)^2} + \frac{2\sfc_2z^2}{(1-z)^3}.
\eeq
The lemma follows from a straightforward computation.
\end{proof}
\begin{lemma}
\label{lem:deg3-polynomial}
For every $\sfu \in \bbR^2$ and $z\in \bbR$, let
\beq
P_\sfu(z) = 2 z^3 + (\sfu_1 - 5)z^2 + (\sfu_2 + 4)z -1.
\eeq
There exists a neighborhood of the origin $\cN\subseteq \bbR^2$ and a function $\sfu\in\cN \mapsto z_0(\sfu)$ such that, for every $u\in \cN$ such that $\sfu_1+ \sfu_2 > 0$, $z_0(\sfu)$ is the only root of $P_\sfu$ in $(-\infty, 1)$. Moreover,
\beq
z_0(\sfu) = \frac12 - \frac{\sfu_1}{2}- \sfu_2 + 2\sfu_1^2 + 6\sfu_2^2 + 7 \sfu_1 \sfu_2 + o(\|\sfu\|^2), \qquad \sfu \to 0.
\eeq
\end{lemma}
\begin{proof}[Proof of Lemma~\ref{lem:deg3-polynomial}]
First, observe that
\beq
P_0(z) = (2z-1)(z-1)^2
\eeq
admits $1/2$ as a simple root. The result follows from the Implicit Function Theorem and an explicit second-order Taylor expansion. Letting
\beq
z_0(\sfu) = \frac12 + \langle(\sfD z_0)(0), \sfu \rangle + \frac12 \langle \sfu, (\sfD^2 z_0)(0) \sfu \rangle + o(\|\sfu\|^2), \qquad \sfu \to 0,
\eeq
(with column vectors) we obtain
\beq
\ba
0 &= P_\sfu(z_0(\sfu))= P_0(z_0(\sfu)) + \sfu_1 z_0(\sfu)^2 + \sfu_2 z_0(\sfu)\\
&= P_0'(\tfrac12)[z_0(\sfu)-\tfrac12] + \tfrac12 P_0''(\tfrac12)[z_0(\sfu)-\tfrac12]^2\\
&\qquad + \tfrac14 \sfu_1 + \sfu_1[z_0(\sfu)-\tfrac12] 
+\tfrac12 \sfu_2 + \sfu_2[z_0(\sfu)-\tfrac12] + o(\|\sfu\|^2)\\
&= \langle \tfrac12 (\sfD z_0)(0) + (\tfrac 14, \tfrac 12)^\intercal, \sfu \rangle\\
&\qquad + \langle \sfu, 
[\tfrac14 (\sfD^2 z_0)(0) - 2(\sfD z_0)(0)(\sfD z_0)(0)^\intercal + (1,1)^\intercal (\sfD z_0)(0)^\intercal]
 \sfu \rangle + o(\|\sfu\|^2).
\ea
\eeq
Therefore,
\beq
(\sfD z_0)(0) =
\left(
\begin{array}{c}
-1/2 \\
-1
\end{array}
\right),
\qquad
(\sfD^2 z_0)(0) =
\left(
\begin{array}{cc}
4&7 \\
7& 12
\end{array}
\right).
\eeq
The restriction that $\sfu_1+ \sfu_2 > 0$ in the statement comes from the fact that $P_\sfu(1) = \sfu_1+ \sfu_2$ and allows us to verify that $z_0(\mathsf{u})$ is indeed the only root in $(-\infty,1)$ rather than in a neighborhood of $1/2$.
\end{proof}
\section{High-temperature expansion of the finite-volume free energy}
In this section we provide the necessary details for the computation behind Remark~\ref{rmk:naive-exp}.
\begin{proposition} 
\label{pr:finite-fe-high-temp}
\label{pr:naive-exp-high-temp} In addition to the $\go_i$'s being i.i.d.\ square integrable random variables with zero mean and unit variance, we assume that $\bbE(\go_1^4) < \infty$.
For every $n\in\bbN$, the finite-volume (averaged) quenched free energy has the following high-temperature ($\gb\to 0$) expansion:
\beq
\frac1n \bbE \log \bar Z_n^{\gb,\go} = - \gb + \mathfrak{c}_n\gb^2 + o(\gb^2), \qquad \text{with} \quad
\lim_{n\to \infty}\mathfrak{c}_n = 4/3.
\eeq
\end{proposition}

\begin{proof}[Proof of Proposition~\ref{pr:naive-exp-high-temp}] Let $n\in \bbN$. Expanding the exponential in~\eqref{eq:bar-pf}, we have
\beq
\bar Z_n^{\gb,\go} = 1 -\gb \sum_{x\in \bbN} \bE[\gO_n(x)^2] + [1+o(1)]\frac{\gb^2}{2} \sum_{x,y\in \bbN}\bE[\gO_n(x)^2\gO_n(y)^2].
\eeq
We now apply the logarithm and expand it up to the second-order term. Using that
\beq
\sum_{x\in \bbN} \bbE\bE[\gO_n(x)^2] = \sum_{x\in\bbN} \ell_n(x) = n,
\eeq
we find thereby:
\beq
\bbE \log \bar Z_n^{\gb,\go} = -\gb n + [1+o(1)] \frac{\gb^2}{2}\mathsf{C}_n,
\eeq
where
\beq
\ba
\mathsf{C}_n &:= \sum_{x,y\in \bbN} \bbE \cov(\gO_n(x)^2, \gO_n(y)^2)\\
&= \sum_{x,y\in \bbN} \bbE \bE^{\otimes 2}[\gO_n(x)^2 \gO_n(y)^2 - \gO_n(x)^2 \tilde \gO_n(y)^2],
\ea
\eeq
and $\tilde S$ is a replica (i.e.\ independent copy) of $S$ (but the fields $\gO_n$ and $\tilde \gO_n$ are built on a common sequence of charges $\go$). Recalling~\eqref{eq:def-Omega-n-x} and expanding all the squares in the above expression, $\mathsf{C}_n$ is found to be equal to: 
\beq
\sumtwo{x,y\in \bbN}{1\le i,j,k,\ell\le n} \bbE(\go_i\go_j\go_k\go_\ell) [\bP(S_i=S_j=x, S_k=S_\ell = y) - \bP(S_i=S_j=x)\bP(S_k=S_\ell = y)].
\eeq
Due to our set of assumptions on the charge sequence $\go$, only the four following cases may contribute to the sum:
\begin{enumerate}
\item\label{case1} $i=j=k=\ell$;
\item\label{case2} $i=j$ and $k=\ell$, but $i\neq k$;
\item\label{case3} $i=k$ and $j=\ell$, but $i\neq j$;
\item\label{case4} $i=\ell$ and $j=k$, but $i\neq j$.
\end{enumerate}
The reader may check that Cases~\ref{case1} and~\ref{case2} give a zero contribution (once we sum over $x$ and $y$) while Cases~\ref{case3} and~\ref{case4} give the same contribution. Therefore, as $n\to \infty$,
\beq
\mathsf{C}_n = 4 \sum_{1\le i < j \le n} [\bP(S_i = S_j) -  \bP(S_i = S_j)^2]\\
\sim 4n \sum_{\ell \ge 1} [2^{-\ell} - 4^{-\ell}] = \frac{8n}{3},
\eeq
which completes the proof.
\end{proof}
\section{Corrections to the undirected model}
\label{sec:corrections}
In this section we correct two propositions from~\cite[Appendix D]{CdHPP16} concerning the \emph{undirected} quenched charged polymer. The statement of the first proposition therein (Proposition D.1) is left unchanged, although we stress that the result holds in any dimension. The proof however must be corrected. In the second proposition (Proposition D.2), the sufficient condition for ballistic behaviour in dimension one is amended.
\par Although this section is independent of the rest of the paper, we prefer to stick to our notation (that is anyway not far from the one used in~\cite{CdHPP16}) in order to make the paper more self-contained. The two major differences with the directed model from Section~\ref{sec:DGH-model} are the following:
\begin{itemize}
\item $\bP$ denotes the law of \emph{simple random walk} $S=(S_i)_{i\in\bbN_0}$ (started at the origin) on $\bbZ^d$, i.e.\ $(S_i-S_{i-1})_{i\in\bbN}$ is a sequence of i.i.d.\ random variables uniformly distributed on the $2d$ unit vectors.
\item $\bbP$ denotes the law of i.i.d.\ real-valued random variables $\go = (\go_i)_{i\in \bbN}$ with zero mean, unit variance and finite exponential moments, i.e.\ $M(\gd):= \bbE(e^{\gd \go_1}) < \infty$ for every $\gd\ge 0$. Here, we also allow \emph{biased} charge distributions: for every $\gd\ge0$, we let $\bbP_\gd$ be the \emph{tilted} probability measure uniquely defined by the property:
\beq
\bbP_\gd(A) = \bbE(e^{\gd \gO_n - n \log M(\gd)} \ind_A), \qquad \forall n\in\bbN,\quad \forall A\in \gs(\go_1,\ldots,\go_n).
\eeq
Note that the random variables $(\go_i)_{i\in \bbN}$ remain i.i.d.\ under $\bbP_\gd$ and that $\bbP_0 = \bbP$.
\end{itemize} 
In addition, we let
\beq
R_n(S) := \card\{S_1, \ldots, S_n\}
\eeq
be the number of distinct vertices visited by the random walk between time $1$ and time $n$.

\par The proposition below corresponds to~\cite[Proposition D.1]{CdHPP16}. The lower bound in the proof must be corrected. Also, we stress that the result is valid in any dimension.
\begin{proposition}
\label{pr:1}
Let $d\ge 1$.
Suppose that $\delta,\beta \in (0,\infty)$. Then there exist $c_1, c_2>0$ 
(depending on $\delta,\beta$) such that, for $\bbP_\delta$-a.e.\ $\omega$,
\begin{equation}
\label{eq:1}
\bar \bP_n^{\beta,\go}(R_n(S) \leq c_1 n) \leq e^{-c_2 n+o(n)}.
\end{equation}
\end{proposition}
\begin{proof}[Proof of Proposition~\ref{pr:1}]
Let $\pi$ be the one-sided path that takes right-steps only, i.e., $\pi_i = (i, 0,\ldots 0)$ for $i\in\N_0$. 
Let us denote by $\bar H_n^{\omega}(S)$ the Hamiltonian in~\eqref{eq:part-fct-conv2} and estimate
\begin{equation}
\label{eq:2}
\bar Z_n^{\beta, \go} \geq \, \bE\Big[e^{-\beta \bar H_n^{\omega}(S)}\,
\ind_{\{S_i=\pi_i\,\forall\,1\leq i\leq n\}}\Big]
= (\tfrac{1}{2d})^n\,e^{-\beta\sum_{i=1}^n \omega_i^2}
= (\tfrac{1}{2d})^n\,e^{-\beta \bbE_\gd(\go_1^2) n + o(n)},
\end{equation}
where the $o(\cdot)$ holds $\bbP_\gd$-a.s, by the law of large numbers.
Moreover, by Jensen's inequality we have (recall~\eqref{eq:bar-pf},~\eqref{eq:def-Omega-n-x} and the definition of $\ell_n$ in~\eqref{eq:wsaw2})
\begin{equation}
\begin{aligned}
\bar H_n^\omega(S) 
&= \sum_{ {x\in\Z\colon} \atop {\ell_n(x)>0} } 
\left(\sum_{i=1}^n \omega_i \ind_{\{S_i=x\}}\right)^2
= R_n(S) \left[\frac{1}{R_n(S)} \sum_{ {x\in\Z\colon} \atop {\ell_n(x)>0}} 
\left(\sum_{i=1}^n \omega_i \ind_{\{S_i=x\}}\right)^2\right]\\
\label{eq:3}
&\geq R_n(S) \left(\frac{1}{R_n(S)} \sum_{ {x\in\Z\colon} \atop {\ell_n(x)>0}}    
\sum_{i=1}^n \omega_i \ind_{\{S_i=x\}}\right)^2
= \frac{1}{R_n(S)}\,\Omega_n^2.
\end{aligned}
\end{equation}
Combining (\ref{eq:2}--\ref{eq:3}), we obtain 
\begin{equation}
\label{eq:4}
\begin{aligned}
\bar \bP_n^{\beta,\omega}(R_n(S) \leq c_1n)
&\leq e^{\beta [\bbE_\gd(\go_1^2) + o(1)]n}\, (2d)^n\, \bE\left[\exp\left\{-\frac{\beta}{R_n(S)}\,\Omega_n^2\right\}\, 
\ind_{\{R_n(S) \leq c_1n\}}\right]\\
&\leq \exp\left\{-\beta n \left[\frac{1}{c_1n^2}\,\Omega_n^2 -\bbE_\gd(\go_1^2) -\tfrac{\log (2d)}{\beta} +o(1) \right]\right\}.
\end{aligned}
\end{equation}
Note that $\bbE_\gd(\go_1^2)= m(\gd)^2 + m'(\gd)$, where $m(\gd) := (\log M)'(\gd)$. By the strong law of large numbers for $\omega$, we have $\lim_{n\to \infty} n^{-1} \Omega_n
=  m(\delta)>0$ for $\bbP_\delta$-a.e.\ $\omega$, and so 
the term between square brackets equals $c_3[1+o(1)]$ with $c_3=(\tfrac{1}{c_1}-1) m(\delta)^2- m'(\gd)
-\tfrac{\log (2d)}{\beta}$. Therefore, by choosing $c_1>0$ small enough so that $c_3>0$, we get 
(\ref{eq:1}) with $c_2=\beta c_3$.
\end{proof}
The following proposition corrects~\cite[Proposition D.2]{CdHPP16}. The correction comes from the changes in the proof of Proposition~\ref{pr:1}.
\begin{proposition}
\label{pr:2}
Assume that $d=1$ and that $\gb,\gd\in(0,\infty)$ satisfy
\beq
\label{eq:suff-cond}
m(\gd)^2 - m'(\gd) > \frac{\log 2}{\gb}.
\eeq
Then, there exists $\epsilon=\epsilon(\delta,\gb)>0$ such that:
\begin{equation}
\lim_{n\to\infty} \bar \bP_n^{\beta,\omega}\big(n^{-1}S_n>\epsilon \mid S_n>0\big) = 1.
\end{equation}
\end{proposition}
\begin{remark} 
When $\bbP(\go_1\in \cdot) = \cN(0,1)$ then $m(\gd)= \gd$ and $m'(\gd)=1$, so that the condition in~\eqref{eq:suff-cond} is equivalent to $\gd^2>1+(\log 2)/\gb$ (that is a non-empty condition).
When $\bbP(\go_1\in \cdot)$ is the uniform probability measure on $\{-1,1\}$ then $m(\gd) = \tanh(\gd)$ and $m'(\gd)=1-\tanh(\gd)^2$, so that the condition in~\eqref{eq:suff-cond} is equivalent to $2\tanh(\gd)^2>1+(\log 2)/\gb$. The latter condition is non-empty if and only if $\gb >\log 2$. A similar remark can be made for all \emph{bounded} charge distributions.
\end{remark}
\begin{proof}[Proof of Proposition~\ref{pr:2}]
Recall the value of $c_3$ in the proof of Proposition~\ref{pr:1}. If~\eqref{eq:suff-cond} is satisfied then we can choose $c_1 > \tfrac12$ in Proposition~{\rm \ref{pr:1}} 
and use the inequality
\begin{equation}
\frac{\left|\{x\in \bbZ\colon\, \ell_n(x) = 1\}\right|}{n} \geq \frac{2R_n(S)}{n} -1
\end{equation}
to conclude that a positive fraction of the sites are visited precisely once. Consequently, if the 
polymer chain chooses to go to the right, then $S_n/n$ has a strictly positive $\liminf$.
\end{proof}

\par For the \emph{undirected} charged polymer with quenched \emph{biased} charges, one concludes that the number of visited vertices is proportional to a positive fraction of the chain, with large probability. This property may be upgraded to true ballistic behavior, at least for a large enough charge bias, with an extra condition on the temperature for some charge distributions. We close the section with two open questions:
\begin{enumerate}
\item Can we prove ballistic behavior for \emph{every} $\gd,\gb\in(0,\infty)$? 
\item What can be said about the case of \emph{neutral} charges ($\gb=0$)?
\end{enumerate}
\bibliographystyle{abbrv}
\bibliography{QuenchedChargedPol.bib}

\end{document}